\documentclass[11pt]{amsart}

\usepackage[margin=3cm]{geometry}
\usepackage{amsmath}
\usepackage{amsthm, array}
\usepackage{amssymb}
\usepackage{mathabx}
\usepackage{amsfonts}
\usepackage{enumerate}
\usepackage{enumitem}
\usepackage[linktocpage=true, pdfencoding=auto, psdextra]{hyperref}
\usepackage{todonotes}

\usepackage{mathtools}
\usepackage{xcolor}
\usepackage{tikz}
\usetikzlibrary{positioning}
\usepackage[all]{xy}
\usepackage[pdftex]{pict2e}
\usetikzlibrary{matrix}

\newtheorem{thm}{Theorem}[section]
\newtheorem{cor}[thm]{Corollary}
\newtheorem{lem}[thm]{Lemma}
\newtheorem{prop}[thm]{Proposition}

\theoremstyle{definition}
\newtheorem{defin}[thm]{Definition}

\theoremstyle{remark}

\newtheorem{rmk}[thm]{Remark}

\newtheorem*{theorem*}{Theorem}

\usepackage{scalerel,stackengine}
\stackMath
\newcommand\reallywidehat[1]{%
	\savestack{\tmpbox}{\stretchto{%
			\scaleto{%
				\scalerel*[\widthof{\ensuremath{#1}}]{\kern-.6pt\bigwedge\kern-.6pt}%
				{\rule[-\textheight/2]{1ex}{\textheight}}
			}{\textheight}%
		}{0.5ex}}%
	\stackon[1pt]{#1}{\tmpbox}%
}
\DeclareMathOperator{\lcm}{lcm}
\DeclareMathOperator{\rank}{rank}  
\begin{document}
	\def\mapright#1{\ \smash{\mathop{\longrightarrow}\limits^{#1}}\ }
	\def\mapleft#1{\ \smash{\mathop{\longleftarrow}\limits^{#1}}\ }
	\def\mapup#1{\Big\uparrow\rlap{$\vcenter {\hbox {$#1$}}$}}
	\def\mapdown#1{\Big\downarrow\rlap{$\vcenter {\hbox {$\ssize{#1}$}}$}}
	\def\mapne#1{\nearrow\rlap{$\vcenter {\hbox {$#1$}}$}}
	\def\mapse#1{\searrow\rlap{$\vcenter {\hbox {$\ssize{#1}$}}$}}
	\def\mapr#1{\smash{\mathop{\rightarrow}\limits^{#1}}}
	\def\lb{[}
	\def\ss{\smallskip}
	\def\at{{\widetilde\alpha}}
	\def\sm{\wedge}
	\def\la{\langle}
	\def\ra{\rangle}
	\def\on{\operatorname}
	\def\ssize{\scriptstyle}
	\def\ar#1{\stackrel {#1}{\rightarrow}}
	\def\br{\bold R}
	\def\bc{\bold C}
	\def\si{\sigma}
	\def\zp{\bold Z_p}
	\def\da{\downarrow}
	\def\xbar{{\overline x}}
	\def\ebar{{\overline e}}
	\def\ni{\noindent}
	\def\coef{\on{coef}}
	\def\den{\on{den}}
	\def\ot{\otimes}
	\def\ms{\medskip}
	\def\cl{ \mathcal{L }}
	\def\cf{\mathcal{F}}
	\def\ct{\mathcal{T}}
	\def\ctl{\mathcal{T}_L}
	\def\d{\displaystyle\mathop}
	\def\wcm{ \widetilde{\cm}}
	\def\supp{\mathrm{supp}}
	\newtheorem{expls}[equation]{Examples}
	\newtheorem{exam}[equation]{Example}
	\newtheorem{exams}[equation]{Examples}
	\def\ia{\emph{em[\rm (a)]\ }}
	\def\ib{\emph{em[\rm (b)]\ }}
	\def\ic{\emph{em[\rm (c)]\ }}
	\def\id{\emph{em[\rm (d)]\ }}
	\def\ad{{\rm ad}}
	\def\im{{\rm im}}
	\def\hl{{ \hat{\ell}}}
	\def\hn{{ \hat{ \mathcal N}}}
	\def\ead{{\rm expad}}
	\def\nz{\mathbb{Z}}
	\def\nk{\Bbbk}
	\def\nn{\mathbb{N}}
	\def\nzp{{\mathbb{Z}^+}}
	\def\nr{\mathbb{R}}
	\def\nc{\mathbb{C}}
	\def\cg{ \mathcal{G }}
	\def\cn{ \mathcal{N }}
	\def\cw{ \mathcal{W }}
	\def\cm{ \mathcal{M }}
	\def\lhn{ \hat{n}}
	\def\pf{\noindent {\bf Proof :\ }}
	\def\glnc{\mathfrak{gl}(n,\nc)}
	\def\glnr{\mathfrak{gl}(n,R)}
	\def\vsmfv{\vspace{-.05in}}
	\def\vspfv{\vspace{.05in}}
	\def\vspsv{\vspace{.07in}}
	\def\pr{^{\,\prime}}
	\def\mh{ \widehat{m}}
	\def\per{.$\!$\, }
	\def\fg{\mathfrak{g}}
	\def\fl{\mathfrak{L}}
	\def\sl{\mathfrak{sl}}
	\def\gl{\mathfrak{gl}}
	\def\ibul\emph{em[$\bullet$]}
	\def\dpr{^{\prime \prime}}
	\def\without{w.l.o.g\per }
	\def\pgs{pp\per }
	\def\skp#1{\vskip#1cm\relax}
	\def\nd{\noindent}
	\def\der{\dfrac{d}{dx}}
	\def\lm{\lim\limits_{h \to 0}}
	\def\dif{\dfrac{d}{dx}}
	\def\vrr{\mathcal{VR}(Q_n;r)}
	\def\vr5{\mathcal{VR}(Q_5;3)}
	\def\Lu{ Lyubeznik }
	\def\wc{\widehat{\mathcal{C}}}
	\def\kv{\Bbbk[x_0,\ldots,x_{m-1}]}
	\def\kx{\Bbbk[x_0,\ldots,x_{2^n-1}]}
	\def\VR{VR(Q_n;r)}
	\def\vp{\varPhi}

	\def\pfbox{	\hspace{5 in} \qedsymbol}
	\setlength{\oddsidemargin}{.25in}
	\setlength{\textwidth}{6in}
	\newcommand*{\backin}{\rotatebox[origin=c]{-180}{$\in$}}

	\setlength{\oddsidemargin}{.25in}
	\setlength{\textwidth}{6in}

	\def\zk{\mathcal{Z}_K}

	\def\VR{VR(Q_n;r)}
	\def\nz{\mathbb{Z}}
	\newcommand{\Tor}{\mathrm{Tor}}

	\title{Cohomological properties of the Vietoris--Rips Complex of a Hypercube Graph}

	\author[M.~Bendersky]{Martin~Bendersky}
	\address{Department of Mathematics, Hunter College, CUNY 695 Park Avenue New York, NY 10065, U.S.A.}
	\email{mbenders@hunter.cuny.edu}
	
	\author[S.~Elia]{Salvatore Elia}
	\address{School of Mathematical Sciences, University of Southampton, SO17 1BJ Southampton, UK}
	\email{S.Elia@soton.ac.uk}
	\author[J.~Grbi\' c]{Jelena~Grbi\' c}
	\address{School of Mathematical Sciences, University of Southampton, SO17 1BJ Southampton, UK}
	\email{J.Grbic@soton.ac.uk}

	
	
	\begin{abstract}
		We develop a toric topological framework for studying the cohomology of Vietoris--Rips complexes $VR(Q_n;r)$ of hypercube graphs. Using total domination invariants and spectral methods, we establish general lower bounds on connectivity, which leads to infinite families of counterexamples to Shukla's conjecture, and derive first global upper bounds on coconnectivity. Our approach interprets Vietoris--Rips complexes via Stanley--Reisner rings, moment-angle complexes, and Tor algebras, allowing global topological information to be extracted from combinatorial data.   In a second direction, we construct explicit cohomology classes using the Koszul resolution and show that they decomposable products of $1$-dimensional classes, and that their representatives can be combimbinatorially realised as the boundary of cross polytopes positively answering the question posed by Adams and Virk. We introduce ghost vertices as a new tool for detecting, extending, and proving linear independence of cohomology classes.
	\end{abstract}
	
	\maketitle

	\section{Introduction}
	Let $Q_n$ be the metric space whose underlying set is the vertex set of the $n$-dimensional hypercube $\mathbb I^n$ equipped with the shortest--path metric. Via the natural identification of $\mathbb I^n=\{0,1\}^n$, each vertex of $Q_n$ corresponds to a binary string of length $n$. Endowed with the \emph{Hamming distance} $d_H$, which counts the number of coordinates in which two binary strings differ, this identification gives an isometry between the graph metric on $Q_n$ and its Hamming metric description.
	
	Given a metric space $(Q_n,d)$ and a scale parameter $r\ge 0$, the \emph{Vietoris--Rips complex} $\VR$ of the $n$-hypercube graph is a simplicial complex with vertex set $Q_n$, in which a finite subset $\sigma\subseteq Q_n$ spans a simplex if and only if its diameter with respect to $d$ is at most $r$.
	
	Vietoris--Rips complexes occupy a central position in the foundation of algebraic topology and geometric group theory, while more recently they have played a prominent role in applied and combinatorial topology, particularly in topological data analysis, where they serve as computable models for recovering the homotopy type of metric spaces. In the highly structured setting of graph metrics and notably for hypercube graphs, they also form a natural testing ground for connections between combinatorics, topology, and commutative algebra.
	
	The topology of $VR(Q_n;r)$ has attracted growing attention in recent years; see, for example, \cite{aa}, \cite{av}, \cite{f}, \cite{s}. 
	
	Adamaskez and Adams~\cite{aa} showed that $VR(Q_n;2)$ has the homotopy type of a wedge of 3-spheres, Shukla~\cite{s} proved that the cohomology of $VR(Q_n;3)$ is concentrated in degrees $4$ and $7$, and  Feng~\cite{f} subsequently identified its homotopy type as a wedge of spheres in these dimensions. Motivated by these results, Shukla conjectured~\cite{s}, that for $n\geq r+2$, the reduced homology $\widetilde{H}_i(VR(Q_n;r);\nz)$ is non-trivial if and only if $i=r+1$ or $r=2^r-1$, which would in particular imply that $VR(Q_n;r)$ is $r$-connected. Using the powerful Easley Cluster at Auburn University and the Ripser software package, Fang~\cite{av} calculated the first $15$ homology groups of $VR(Q_6;4)$ with $\mathbb{Z}/2$ coefficient and found that the first non-tivial homology group appears in degree 7,
	which shows that $VR(Q_6;4)$ is $6$-connected. 
	
	In the present work, we developed a toric topological framework for the study of Vietoris--Rips complexes, centered around Stanley--Reisner rings, and Tor algebras. 
	
	The first main goal of this paper is to study \emph{global topological properties} of Vietoris--Rips complexes of hypercubes, with emphasis on connectivity and its dual notion, coconnectivity. Our approach is based on an interpretation of the Vietoris--Rips complexes as independence complexes of associated graphs, allowing us to combine methods from graph theory, spectral theory, and commutative algebra. Using total domination invariants and Laplacian eigenvalue estimates, we obtain general lower bounds on connectivity, producing infinite families of counterexamples to Shukla’s conjecture. On the other hand, by analysing the Tor algebra of the associated Stanley--Reisner ring, and establishing a new duality of subcomplexes of the Taylor complex, we derive the first general upper bounds on the coconnectivity of $VR(Q_n;r)$, sharply restricting the range of degrees in which cohomology may appear.
	
	The second main focus of the paper is \emph{constructive and algebraic}. Building on ideas announced in~\cite{be} and developed further by Adams and Virk~\cite{av}, we give an explicit description of cohomology classes in Vietoris--Rips complexes using tools from toric topology, including moment-angle complexes, the Koszul resolution, and Hochster’s decomposition of Tor algebras. From this perspective, cohomology classes correspond to concrete monomials in the Koszul
	complex and can be interpreted topologically via full subcomplexes of
	$VR(Q_n;r)$. In particular, we show that the classes constructed by Adams and Virk arise naturally as $*$-decomposable elements in the cohomology of the associated moment-angle complex, providing a conceptual explanation of their structure and propagation behavior. Moreover, we show that their representatives can be combimbinatorially realised as the boundary of cross polytopes positively answering the question posed by Adams and Virk~\cite{av}.
	
	A key innovation of this paper is the introduction of \emph{ghost vertices}. This construction allows us to enlarge full subcomplexes while retaining precise control over the associated Stanley--Reisner rings, Koszul complexes, and Tor algebras. Ghost vertices enable a robust new method for detecting non-trivial cohomology classes and for proving linear independence of families of classes in $H^*(VR(Q_n;r))$. In particular, they allow us to reprove and extend propagation results without recourse to cubical hull conditions, imposed by Adams and Virk~\cite{av}, placing these phenomena in a unified toric topological framework.

	\section{Connectivity of $\VR$} 
	
	In this section, we provide a method for obtaining lower bounds on the connectivity of the Vietoris--Rips complexes $\VR$ for arbitrary values of $n$ and $r$. Our approach is based on the total domination invariant of a graph and leads to  an infinite family of counterexamples to Shukla’s conjecture.
	
	\subsection{Total domination number and connectivity of independence complexes}
	
	We begin by recalling several basic definitions and results concerning total dominating sets of graphs and the independence complexes associated with graphs.
	
	A \emph{graph} $G = (V(G), E(G))$ consists of a nonempty set $V(G)$ of vertices together with a (possibly empty) set $E(G)$ of unordered pairs of distinct vertices, called edges. Throughout this paper, we restrict attention to  simple graphs, that is, graphs without loops or multiple edges. A vertex is said to be \emph{isolated} if it is not incident to any edge.
	
	\begin{defin}[\cite{hy}]
		Let $G$ be a graph with no isolated vertices.
		A \emph{total dominating set} of $G$ is a subset $S\subset V(G)$ such that every vertex of $G$ is adjacent to at least one vertex in $S$. 
	\end{defin} 
	
	We emphasise that each vertex in a total dominating set must itself be adjacent to another vertex in the set. In particular, a vertex does not dominate itself. For example, in the graph illustrated in Figure~\ref{fig:tds}, the set $S=\{1,2\}$ is not a total dominating set, since the vertex $1$ is not adjacent to any vertex in $S$. In contrast, the set $S=\{0,2\}$ is a total dominating set, as every vertex of $G$ is adjacent to at least one vertex in $S$. Finally, the singleton set $\{0\}$ is not a total dominating set, since vertex $0$ is not adjacent to itself.
	\begin{figure}
		\centering
		\label{fig:tds}
		\setlength{\unitlength}{1cm}
		\begin{picture}(3,1.5)
			\thicklines
			\put(0,0.5){\circle*{0.1}}
			\put(-0.3,0.3){$0$}
			
			\put(2,0.5){\circle*{0.1}}
			\put(2.1,0.3){$1$}
			
			\put(1.8,1.1){\circle*{0.1}}
			\put(2,1){$2$}
			
			\put(0,0.5){\line(1,0){2}}
			\put(0,0.5){\line(3,1){1.8}}
		\end{picture}	
		\caption{Total Dominating Sets}
	\end{figure}
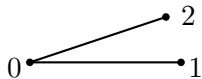
	\begin{defin}
		The \emph{total domination number} of a graph $G$, denoted by $\gamma_t(G)$, is the minimum cardinality of a total dominating set of $G$.
	\end{defin} 
	Equivalently, $\gamma_t(G)$ is the smallest size of a subset of vertices such that every vertex of $G$ has a neighbour in this subset. Our interest in the invariant $\gamma_t(G)$ is motivated by a theorem of Chudnovsky~\cite{c}, which relates the total domination number of a graph to the connectivity of its independence complex.
	
	Let $G=(V,E)$ be a graph. The \emph{complement} of $G$, denoted by $G^c=(V, E^c)$, is the graph with the same vertex set $V$, where two distinct vertices are adjacent in  $G^c$ if and only if they are not adjacent in $G$. The \emph{independence complex} of $G$, denoted by $I(G)$, is defined as the clique complex of the complement graph $G^{c}$.
	
	It is straightforward to observe that the Vietoris--Rips complex $\VR$ is the clique~complex of its 1-skeleton. For integers $n$ and $r$, let $G_{n,r}$ denote the 1-skeleton of $\VR$; that is, $G_{n,r}$ is the graph with vertex set $Q_n$, in which two vertices $a,b\in Q_n$ are adjacent if and only if $d_H(a,b)\le r$. Consequently, the Vietoris--Rips complex $\VR$ can be identified with the independence complex of the complementary graph $G_{n,r}^{c}$, namely, $\VR = I\bigl(G_{n,r}^{c}\bigr)$.
	
	Throughout this paper, the connectivity of a simplicial complex $K$ is defined as the connectivity of its geometric realisation $|K|$. More precisely, $K$ is said to be \emph{$k$-connected} if $\pi_i(|K|)=0$ for all $0\le i\le k$, and we write $\mathrm{Conn}(K)\ge k$. Assuming that $K$ is simply connected, this holds if and only if $\widetilde H_i(K;\mathbb{Z})=0$ for all $i\le k$ and therefore $\mathrm{Conn}(K)=\max\{k\ |\ \widetilde H_i(K;\mathbb{Z})=0 \text{ for all } i\le k\}$.
	The connectivity of an independence complex can be controlled by the total domination number of the graph, as made precise by the following theorem.
	
	\begin{thm}[\cite{c}, \cite{m}] \label{thm:ch} 
		If $\gamma_t(G) > 2k$, then the independence complex $I(G)$ is $(k-1)$-connected.
	\end{thm}
	
	Computing the total domination number of a graph is in general difficult, with exact values known only in special cases. Consequently, one often relies on general estimates in terms of basic graph invariants.
	Recall that the \emph{order} of a graph is the number of its vertices, while the \emph{degree} of a vertex is the number of edges incident to it. The maximum degree of a vertex in a graph $G$ is denoted by $\Delta(G)$.
	We will use the following crude estimate of the total domination number to derive a lower bound on the connectivity of $\VR$.
	
	\begin{thm}[\cite{hy}Theorem 2.11]\label{thm:hy}
		If $G$ is a graph of order $m$ with no isolated vertices, then 
		\[
		\gamma_t(G) \geq \frac{m}{\Delta(G)}.
		\]
	\end{thm}
	
	The estimate in Theorem \ref{thm:hy} follows directly from the definition of a total dominating set. Namely, every vertex of $G$ lies in the open neighborhood of at least one vertex in a total domination set. Consequently, a total dominating set of cardinality $\gamma_t(G)$ can collectively dominate at most $\gamma_t(G)\,\Delta(G)$ vertices, implying the inequality $\gamma_t(G)\,\Delta(G)\ge m$. Equality holds precisely when each vertex is adjacent to exactly one vertex of the total dominating set. An example illustrating this extremal case is shown in Figure~\ref{fig:tds2}. In this graph, the total dominating set $\{S_1,\ldots,S_{\gamma_t(G)}\}$ consists of vertices each having degree $\Delta(G)$. The graph decomposes into isomorphic components, each containing two vertices from the dominating set and $2\Delta(G)$ vertices in total. The inequality in Theorem~\ref{thm:hy} becomes strict if some of the termini of the vertical edges coincide.
	
	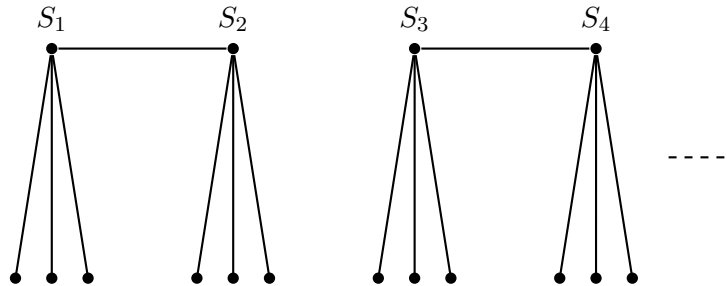
\begin{figure}[ht]
		\centering
		\begin{tikzpicture}[
			scale=0.8,
			thick,
			vertex/.style={circle,fill,inner sep=1.5pt},
			]
			\node[vertex,label=above:$S_1$] (S1) at (0,1.8) {};
			\node[vertex,label=above:$S_2$] (S2) at (3,1.8) {};
			\node[vertex,label=above:$S_3$] (S3) at (6,1.8) {};
			\node[vertex,label=above:$S_4$] (S4) at (9,1.8) {};
			
			\foreach \x in {-0.6,0,0.6} {
				\node[vertex] at (\x,-2) {};
				\draw (S1) -- (\x,-2);
			}
			
			\foreach \x in {2.4,3,3.6} {
				\node[vertex] at (\x,-2) {};
				\draw (S2) -- (\x,-2);
			}
			
			\foreach \x in {5.4,6,6.6} {
				\node[vertex] at (\x,-2) {};
				\draw (S3) -- (\x,-2);
			}
			
			\foreach \x in {8.4,9,9.6} {
				\node[vertex] at (\x,-2) {};
				\draw (S4) -- (\x,-2);
			}
			
			\draw (S1) -- (S2);
			\draw (S3) -- (S4);
			
			\draw[dashed] (10.2,0) -- (11.2,0);
			
		\end{tikzpicture}
		\caption{A graph achieving equality in the total domination bound.}
		\label{fig:tds2}
	\end{figure}
	
	Since the Vietoris--Rips complex $\VR$ is the independence complex of the complementary graph $G_{n,r}^c$, we apply Theorems \ref{thm:ch} and \ref{thm:hy} to the graph $G_{n,r}^c$. We start by calculating the degree of a vertex in $G_{n,r}^c$. 
	
	\begin{lem}\label{lem:order} 
		The degree of any vertex of $G_{n,r}^c$ is
		\[
		\sum_{i=r+1}^n \binom{n}{i}. 
		\]
	\end{lem}
	
	\begin{proof}
		Fix a vertex $v \in Q_n$. For each integer $i$, there are exactly $\binom{n}{i}$ vertices $w\in Q_n$ such that $d_H(v,w)=i$. By the definition of $G_{n,r}^{c}$, a vertex $w$ is adjacent to $v$ if and only if $d_H(v,w)\ge r+1$. Summing over all such values of $i$ gives the stated formula.
	\end{proof}
	
	\begin{thm} \label{thm:main} {\rm
			Let 
			\[\alpha_{n,r} = \dfrac{2^{n-1}}{\sum_{i=r+1}^n \binom{n}{i} }.\]
			Then the Vietoris--Rips complex $\VR$ is $(k-1)$-connected, where
			\[
			k=	\begin{cases}
				\Big[ \alpha_{n,r} \Big] & \mbox{ if } \alpha_{n,r}\notin\mathbb Z \\
				\alpha_{n,r}  -1 & \mbox{ if }  \alpha_{n,r}\in\mathbb Z.
			\end{cases}
			\]
		}
	\end{thm}
	\begin{proof}
		By Theorem \ref{thm:hy}, together with Lemma~\ref{lem:order}, we obtain
		\[
		\gamma_t(G_{n,r}^c) \geq \dfrac{2^{n}}{\sum_{i=r+1}^n \binom{n}{i} }=2  \alpha_{n,r} > 2k.
		\] 
		It follows from Theorem \ref{thm:ch} that  $\VR=I(G^c_{n,r})$ is $(k-1)$-connected.
	\end{proof}
	
	The connectivity bound given in Theorem~\ref{thm:main} is  coarse. For instance, if $r$ is fixed, the lower bound on the connectivity tends to zero as $n\to\infty$. Nevertheless, the estimate is strong enough to provide an infinite family of examples for which the connectivity of $\VR$ exceeds $r$. Moreover, these bounds suggest that the connectivity of $\VR$ does not grow linearly with respect to either $n$ or $r$, and they illustrate that for certain values of $n$ and $r$, Vietoris--Rips complexes are highly connected.  
	
	To illustrate this behavior, we list several explicit examples of the connectivity bounds obtained from Theorem~\ref{thm:main}.
	
	\begin{table}[ht]
		\centering
		\begin{tabular}{c c c }
			\hline
			$n$ & $r$ & lower bound on connectivity \\
			\hline
			7  & 5  & 6   \\
			8  & 6  & 13  \\
			9  & 7  & 24  \\
			12 & 10 & 156 \\
			18 & 15 & 761 \\
			18 & 16 & 6897 \\
			20 & 16 & 387 \\
			20 & 18 & 24964 \\
			\hline
			&&\\
		\end{tabular}
		\caption{Lower bounds on the connectivity of $\VR$}
		\label{table:connectivity}
	\end{table}   
	
	We emphasise that the connectivity bound of Theorem~\ref{thm:main} does not contradict the known non-triviality of $H_{2^r-1}(\VR;\nz)$ established in~\cite{av}.  
	
	The large connectivity values obtained here arise from a very crude estimate of the total domination number. We expect that the actual total domination numbers of the graphs $G^c_{n,r}$ are often significantly larger than those predicted by Theorem~\ref{thm:hy}. For example, computer calculation of Feng~\cite{av}, shows that $VR(Q_6;4)$ is exactly $6$-connected, that is,  that $H_7(VR(Q_6; 4);\nz) \neq 0$.  This implies that $\gamma_t(G^c_{6,4})\leq 16.$ In contrast, Theorem~\ref{thm:main} implies only the weaker conclusion that $\gamma_t(G^c_{6,4})\geq 8$ and hence that $VR(Q_6;4)$ is at least $3$-connected. It would be interesting to determine how close the total domination number $\gamma_t(G^c_{6,4})$ is to $16$.
	
	In general, describing total dominating sets of an arbitrary graph is a difficult problem. However, hypercube graphs have a high degree of symmetry and one might hope that the geometry of hypercubes and their duals, the cross polytopes, could shed light on particularly nice structures of total dominating sets in this setting. We conclude this section by describing a connection between certain non-trivial homology classes in $H_*(\VR;\nz)$ and total dominating sets.
	
	\begin{prop} 
		Suppose a non-trivial homology class $a \in H_{m-1}(\VR;\nz) $ is represented by a cycle $\alpha$, which is the boundary of a cross polytope on $2m$ vertices.  Then the set of vertices of $\alpha$ is a total dominating set of $G_{n,r}^c$.
	\end{prop}
	\begin{proof}
		Let the vertices of $\alpha$ be $\mathcal{C}=\{ v_1,w_1,v_2,w_2, \ldots, v_m,w_m\}$, where  $\{v_i,w_i\}$ is not an edge in $\VR$, while every other pair of vertices in $\mathcal C$ spans an edge. Equivalently, the vertices $\{v_i,w_i\}$ in $\mathcal C$ represent a matching edge set in $G_{n,r}^c$. If $\mathcal{C}$ were not a total dominating set, then there would be a vertex $u$ not adjacent to any vertex of $\mathcal{C}$ in $G_{n,r}^c$. This would imply that $u *\alpha$ is a subcomplex of $\VR$, contradicting the assumption that $\alpha$ represents a non-trivial homology class. Hence $\mathcal C$ is a total dominating set of $G_{n,r}^c$.
	\end{proof}

	\subsection{Graph Laplacian}
	The \emph{graph Laplacian} provides a spectral framework for estimating the domination type invariants of graphs, including the total domination number. Although computing $\gamma_t(G)$ is NP-hard in general, spectral methods often provide effective bounds that depend only on eigenvalues of natural matrices associated with $G$. These bounds are particularly appealing in highly symmetric or large graphs, where exact combinatorial methods are infeasible.
	
	The Laplacian matrix of a graph encodes both local and global structural information. Its eigenvalues reflect fundamental properties such as vertex degrees, expansion, and connectivity, and they have been successfully used to bound a variety of graph parameters, including independence, domination, and matching numbers. In particular, several results relate domination-type invariants to the largest Laplacian eigenvalue and to algebraic connectivity, linking domination theory to spectral graph theory (see, for example,\cite{chung, biggs}).

	\begin{defin}
		Let $G=(V(G),E(G))$ be a connected graph with vertex set $V(G)=\{v_1, \dots, v_n\}$. The \emph{Laplacian} of $G$ is the $n\times n$ matrix $L(G)=(L_{ij})$ defined by
		\[
		L_{ii}=d_i \quad \text{and} \quad
		L_{ij}=
		\begin{cases}
			-1 & \text{if } ij\in E(G)\\
			0 & \text{if } ij\notin E(G)
		\end{cases}
		\]
		where $d_i$ denotes the degree of the vertex $v_i$.
	\end{defin}
	Aharoni, Berger and Meshulam~\cite[Theorem 4.1]{abm} established a lower bound on the connectivity of the independence complex $I(G)$ in terms of the largest eigenvalue of the Laplacian matrix. Writing $\lambda(G)$ for the largest eigenvalue of $L(G)$, they proved
	\begin{equation}
		\label{eq:abm}
		\mathrm{Conn}(I(G)) \ge \frac{|V(G)|}{\lambda(G)} - 2.
	\end{equation}
	
	Anderson and Morley~\cite{AM} showed that $\lambda(G) \leq 2\Delta(G)$ (see also~\cite[Theorem 3.1]{z}). 
	Substituting this estimate into~\eqref{eq:abm} gives
	\[
	\mathrm{Conn}(I(G))  \geq \frac{|V(G)|}{\lambda(G)} - 2
	\,\geq\, \frac{|V(G)|} {2 \Delta({G})} -2.
	\]
	This bound is implied by Theorem~\ref{thm:ch}, which relates connectivity directly to the total domination number. However, when the Laplacian spectrum of $G$ can be computed explicitly, the inequality~\eqref{eq:abm} often yields strictly stronger connectivity estimates than those obtained from degree considerations alone.
	
	\begin{exam}
		Consider the Vietoris--Rips complex $VR(Q_7;5)$. By Theorem~\ref{thm:main}, the total domination estimate yields a lower bound
		\[
		\mathrm{Conn}\bigl(VR(Q_7;5)\bigr) \ge 6
		\]
		(see Table~\ref{table:connectivity}).  
		Since $VR(Q_7;5)=I(G^c_{7,5})$, we may alternatively calculate the spectral bound. The graph $G^c_{7,5}$ has vertex set
		$Q_7$, hence $|V(G^c_{7,5})|=2^7=128$, and its Laplacian matrix has largest
		eigenvalue $\lambda(G^c_{7,5})=14$. Substituting into~\eqref{eq:abm} gives
		\[
		\mathrm{Conn}\bigl(VR(Q_7;5)\bigr)
		=
		\mathrm{Conn}\bigl(I(G^c_{7,5})\bigr)
		\ge \frac{128}{14}-2 > 7.
		\]
		Therefore $VR(Q_7;5)$ is at least $8$-connected.
	\end{exam}

	\section{toric topological Constructions for Simplicial Complexes}

	Simplicial complexes can be studied through a variety of complementary structures, ranging from topological invariants such as homology and homotopy, to combinatorial features such as links and missing faces, and to algebraic objects encoding their face structure. 
	
	Toric topology provides a unifying framework that naturally integrates these topological, combinatorial, algebraic, and geometric perspectives. By functorially associating spaces and algebraic invariants to simplicial complexes, toric topology allows one to study their topology through combinatorial data and to interpret algebraic constructions, such as Stanley--Reisner rings, as cohomological invariants of canonically associated spaces. This interplay makes toric topology especially well suited for the study of highly structured simplicial complexes, such as Vietoris--Rips complexes of the hypercube.
	
	We start by recalling   the Stanley--Reisner ring, a fundamental algebraic construction that assigns to each simplicial complex a graded commutative ring whose defining ideal reflects the combinatorics of the complex.
	
	Let $K$ be a simplicial complex with vertex set $[m]=\{0,1,\ldots,m-1\}$. The \emph{Stanley--Reisner ring} (or face ring) of $K$ over a commutative ring $\Bbbk$ with unity is defined as
	\[
	SR[K]
	=
	\Bbbk[x_0,\ldots,x_{m-1}]/\mathcal{I}_K
	\]
	where the Stanley--Reisner ideal $\mathcal{I}_K$ is generated by the square-free monomials
	\[
	x_I = x_{i_1}\cdots x_{i_t}
	\quad\text{for all subsets } I=\{i_1,\ldots,i_t\}\subseteq[m] \text{ with } I\notin K.
	\]
	Throughout this paper, we regard $SR[K]$ as a graded ring by assigning degree $|x_i|=2$ with each generator.
	
	From a combinatorial viewpoint, the Stanley--Reisner ideal is completely determined by the minimal missing faces of $K$. In particular, if $K$ is the clique complex of a graph, then all minimal non-faces are edges, and the ideal $\mathcal{I}_K$ is generated by quadratic monomials.

	Building on the algebraic description of simplicial complexes provided by Stanley--Reisner rings, we now introduce the principal toric topological
	construction that will be used throughout the paper. With each simplicial complex $K$, in toric topology, we  associate a canonical topological space, the \emph{moment-angle complex} $\mathcal Z_K$. This construction provides a
	bridge between the combinatorics of $K$, the algebraic properties of its Stanley--Reisner ring, and the topology of an explicitly defined space. A fundamental feature of this construction is that the cohomology of $K$ appears naturally as a direct summand of the cohomology of $\mathcal Z_K$.  As a
	consequence, questions about the cohomology of simplicial complexes can often
	be approached by analysing the richer algebraic and topological structure of
	their associated moment-angle complexes. In this section we recall the
	definition of $\mathcal Z_K$ and the algebraic description of its cohomology in
	terms of Tor algebras and full subcomplexes of $K$, following the work of
	Hochster, Buchstaber--Panov, and Baskakov (see~\cite{bp} for more details).
	
	Only in the following section will we specialise these constructions to Vietoris--Rips complexes of the hypercube, where the combinatorial structure of $K$ allows for a more explicit description of the Stanley--Reisner ideal and
	effective control of the associated Tor algebra.

	Let $K$ be a simplicial complex on the vertex set $[m]=\{0, \dots, m-1\}$.  The \emph{moment-angle  complex} of $K$, denoted by $\mathcal{Z}_K$, is defined as the colimit
	\[ 
	\zk=\bigcup_{\sigma \in K}(D^2, S^1)^\sigma \subset \prod_{i=0}^{m-1}D^2,
	\quad \text{where }(D^2,S^1)^\sigma = \prod_{i=0}^{m-1} Y_i, \quad  Y_i = \begin{cases} D^2 & \text{for }i \in \sigma \\ S^1 & \text{for }i \notin \sigma. \end{cases}
	\]
	
	Buchstaber and Panov~\cite{bp2}, and independently Franz~\cite{franz} showed that the
	cohomology of a moment-angle complex admits a purely algebraic description in
	terms of the Tor algebra of the Stanley--Reisner ring. More precisely, one has an isomorphism
	\[
	H^{2j-i}(\zk; \Bbbk)\cong \Tor^{-i,2j}_{\Bbbk[x_0, \cdots, x_{m-1}]}(SR[K],\Bbbk)
	\]
	where each generator $x_i$ is assigned degree $2$.

	For a subset $J\subset [m]$, let $K_J$ denote the full subcomplex of $K$ on the vertex set $J$, that is, the simplicial complex consisting of all simplices of $K$ with vertices in $J$. Back in 1977, in his seminal work, Hochster~\cite[Theorem 5.2]{Ho} identified the $\Bbbk$-module decomposition of the Tor algebra as
	\[
	\Tor^{-i}_{\Bbbk[x_0, \ldots,x_{m-1}]}(SR[K],\Bbbk) \cong  \bigoplus_{\substack{J\subset[m]\\|J|=j}}\widetilde{H}^{j-i-1}(|K_J|;  \Bbbk).
	\]
	This result was later refined by Baskakov~\cite{baskakov}, who showed that the direct sum
	$\bigoplus_{J\subset[m]} \widetilde H^{\ast}(|K_J|;\Bbbk)$ carries a natural ring
	structure. The product is defined by the $*$-pairing
	\begin{equation}\label{starp} 
		\widetilde{H}^{|J|-i-1}(|K_J|; \Bbbk) \otimes \widetilde{H}^{|L|-s-1}(|K_L|; \Bbbk) \mapright{\star} \widetilde{H}^{|J|+|L|-(i+s)-1}(|K_{J \cup L}|; \Bbbk)
	\end{equation} 
	induced by the simplicial inclusion $K_{J \cup L} \longrightarrow K_J \ast K_L$ when $J \cap L = \emptyset$ and zero otherwise. By convention, $\Sigma K_{\emptyset}=S^0$.     
	
	Combining these results (see, for example, \cite[Theorems~3.2.9 and 4.5.8]{bp}), the cohomology ring of the moment-angle complex can be expressed in terms of full subcomplexes of $K$ as
	\[ 
	H^{\ast}(\zk;\Bbbk) \cong \bigoplus_{J \subset [m]} \widetilde{H}^{\ast}(|K_J|;\Bbbk).
	\]	
	
	In this paper, we focus both on the Hochster decomposition and on the explicit calculation of the Tor algebra $\Tor^{-i,2j}_{\Bbbk[x_0, \dots, x_{m-1}]}(SR[K],\Bbbk)$.
	To this end, we recall the Koszul resolution of $\Bbbk$ as a module over $\Bbbk[x_0,\ldots,x_{m-1}]$, where $\Bbbk$ is equipped with the trivial module structure, that is, each $x_i$ acts trivially on $\Bbbk$.
	
	Let $\Lambda(u_0,\ldots,u_{m-1})$ denote the exterior algebra generated by $\{u_0,\ldots,u_{m-1}\}$. The \emph{Koszul resolution} is the bigraded differential algebra
	\[
	\mathcal K
	=
	\Lambda(u_0,\ldots,u_{m-1})
	\otimes_{\Bbbk}
	\Bbbk[x_0,\ldots,x_{m-1}]
	\]
	with bidegrees
	\[
	\mathrm{bideg}(u_i)=(-1,2),
	\qquad
	\mathrm{bideg}(x_i)=(0,2)
	\]
	and differential defined by
	\[
	d(u_{i_1}\cdots u_{i_s}\otimes x_{j_1}\cdots x_{j_t})
	=
	\sum_{\ell=1}^s
	(-1)^{\ell+1}
	u_{i_1}\cdots\widehat{u}_{i_\ell}\cdots u_{i_s}
	\otimes
	x_{i_\ell}x_{j_1}\cdots x_{j_t}
	\]
	where $\widehat{u}_{i_\ell}$ indicates omission.
	Together with the augmentation map $\epsilon\colon\mathcal{K}\to \Bbbk$, the Koszul resolution is a resolution of $\Bbbk$.
	
	The associated \emph{Koszul complex} is the bigraded differential algebra
	\[
	\mathcal{K} \otimes_{\kv} SR[K] \cong \Lambda(u_0,\dots,u_{m-1}) \otimes SR[K]
	\]
	with the differential induced by the differential in the Koszul resolution. Its cohomology computes the Tor algebra. In particular, there is an isomorphism of bigraded algebras
	\[
	\Tor^{t,j}_{\kv}(SR[K],\Bbbk) \cong H^{t,j}(\Lambda(u_0,\dots,u_{m-1})\otimes SR[K])
	\]
	which agrees, up to sign, with the $*$-product described in~\eqref{starp}.
	
	For a generator of the Koszul complex
	\[
	\omega=u_{j_1}\cdots u_{j_t} \otimes x_{i_1}\cdots x_{i_s}
	\]
	we define the \emph{support} of $\omega$, denoted by $\supp(\omega)$, to be the subset
	\[
	\supp(\omega)=\{j_1, \ldots, j_t, i_1, \ldots, i_s\} \subset [m].
	\]
	
	For a fixed subset $J\subset[m]$, we define the \emph{Koszul complex with support $J$} to be the subcomplex
	\[
	\Lambda(u_0,\dots,u_{m-1}) \otimes SR[K]\Big|_J
	\]
	of $\Lambda(u_0,\ldots,u_{m-1})\otimes SR[K]$ generated by all elements whose support is exactly $J$.

	As shown in the proof of~\cite[Theorem 3.2.9]{bp}, the cohomology of this subcomplex is naturally identified with the reduced cohomology of the suspension of the full subcomplex $K_J$. More precisely, one has an isomorphism
	\[ 
	H^{|J|-t} (\Sigma K_J; \Bbbk) \cong  H^{-t,2|J|} (\Lambda(u_0,\dots,u_{m-1}) \otimes SR[K]\Big|_J).
	\]
	
	In particular, taking $J=[m]$ and noting that $K_{[m]}=K$, we have
	\[
	\widetilde{H}^{m-t}(\Sigma K; \Bbbk) \cong H^{-t,2m}(\Lambda(u_0,\ldots,u_{m-1}) \otimes SR[K]\Big|_{[m]}).
	\]
	Consequently, a cycle in the $t$th stage of the Koszul complex with full support $[m]$, represented by an element of the form
	\[
	u_{j_1}\cdots u_{j_t}\, x_{i_1}\cdots x_{i_s} + \cdots
	\]
	represents a cohomology class in $\widetilde{H}^s(\Sigma K; \Bbbk)$ as $m-t=s$. In particular, the number of variables $x_i$ appearing in such a representative coincides with the cohomological degree of the corresponding class in the suspension $\Sigma K$.

	\section{Coconnectivity of $\VR$}\label{coconnectivity}
	
	In the previous sections we established lower bounds on the connectivity of Vietoris--Rips complexes of the hypercube, showing in particular that these complexes can be highly connected for a wide range of parameters. In this section we study the complementary problem of bounding cohomology from above.
	
	\begin{defin}
		Let $K$ be a simplicial complex. We define the \emph{coconnectivity} of $K$ to be
		the largest integer $d$ such that
		\[
		H^{d-1}(K;\mathbb{Z}) \neq 0.
		\]
		Equivalently, $K$ has coconnectivity at most $d$ if
		\[
		H^{i}(K;\mathbb{Z}) = 0 \quad\text{for all } i\ge d.
		\]
		Thus coconnectivity measures the top degree in which cohomology is non-trivial,
		or, equivalently, the cohomological dimension of $K$.
	\end{defin}
	
	For Vietoris--Rips complexes $\VR$, lower bounds on connectivity were obtained earlier using domination invariants and spectral methods. By contrast, much less is known about upper bounds on coconnectivity. Prior to this work, the only general bounds arose from the dimension of the simplicial complex itself, since cohomology necessarily vanishes above the dimension of a simplicial complex. A classical result of Kleitman~\cite{k} shows that for any simplex
	$\sigma \in VR(Q_n;r)$,
	\begin{equation}\label{dim}
		|\sigma| \le
		\begin{cases}
			\sum_{i=0}^t \binom{n}{i} & \text{if } r=2t,\\[1ex]
			2 \sum_{i=0}^t \binom{n-1}{i} & \text{if } r=2t+1.
		\end{cases}
	\end{equation}
	While effective, these bounds are coarse and do not reflect finer homotopy theoretical structure of $\VR$.
	
	The main goal of this section is to obtain substantially sharper upper bounds on
	the coconnectivity of Vietoris--Rips complexes by exploiting the algebraic
	structure of their associated Stanley--Reisner rings and the resulting Tor
	algebras.
	
	\begin{thm}\label{thm:coconnect}   
		Let $n>r+1$. Then 
		\[
		H^{\ast}(\VR; \nz) = 0\quad \text{for all } \ast\ge  2^{n-1}-\frac{2^{n-2}S}{2S-1}-1
		\]
		where    
		\[
		S=\sum_{i=r+1}^{n-1} \binom{n}{i}.
		\]
	\end{thm}
	To indicate the scale of improvement provided by
	Theorem~\ref{thm:coconnect}, note that a convenient approximation of the bound
	$2^{n-1}-\frac{2^{n-2}S}{2S-1}-1$
	is $3\cdot 2^{n-3}-1$, which is significantly smaller than the maximal simplex dimension predicted by~\eqref{dim} for the same parameters $n$ and $r$.\todo[inline]{for $n$ close to $r$, as the dim has a polynomial growth for fixed $r$, while the estimate has exponential growth}
	\todo[inline]{I think comparing the growth for fixed $r$ is misleading here -- that's the case when $r$ and $n$ are far apart, whereas the bound from Theorem~4.2 is an improvement when $n$ and $r$ are close together. It might be better to replace this with the statement that the dimension of the simplicial complex ?}

	Table~\ref{table:bound} compares the dimension, the bound obtained from Theorem~\ref{thm:coconnect}, and the largest currently known degree in which cohomology is non-trivial, namely $2^r-1$, for the Vietoris--Rips complexes $VR(Q_{r+2};r)$.
	
	\begin{table}[ht]
		\centering
		
		\begin{tabular}{c c c c}
			\hline
			$r$ 
			& $\dim VR(Q_{r+2};r)$ 
			& Theorem~\ref{thm:coconnect} 
			& $2^{r}-1$ \\
			\hline
			8  & 386      & 377      & 255 \\
			10 & 1586     & 1513     & 1023 \\
			18 & 431\,910 & 389\,850 & 262\,143 \\
			\hline
			&&&\\
		\end{tabular}
		\caption{Upper bounds on the coconnectivity of $VR(Q_{r+2};r)$}
		\label{table:bound}
	\end{table}	
	We now turn to the algebraic machinery required to prove Theorem~\ref{thm:coconnect}. This involves a detailed analysis of the Tor algebra associated to the Stanley--Reisner ring of $\VR$, using the Taylor and Lyubeznik resolutions.

	Recall the Koszul complex $(\mathcal{K},d)$ associated to the Stanley--Reisner ring $SR[K]$. To streamline degree bookkeeping in what follows, it is convenient to reindex
	the bidegrees in the Stanely--Reisner ring (and therefore in the Koszul complex) introduced earlier. Concretely, we replace the generators $v_i$ by generators $x_i$, $i=0,\ldots,m-1$, of bidegree $(0,1)$, and replace the generators $w_i$ by generators $u_i$, $i=0,\ldots,m-1$, of bidegree $(-1,1)$. This produces a bigraded differential algebra which is (non-degree-preservingly) isomorphic to the original Koszul complex $(\mathcal{K},d)$. With this convention, the total internal degree counts the number of generators in a monomial, while the homological degree records the Koszul resolution degree. For simplicity of notation, we continue to denote the regraded Koszul complex and the Stanley--Reisner ring by $(\mathcal{K},d)$ and $SR[K]$, respectively.
	
	We now describe the Stanley--Reisner ring of the Vietoris--Rips complexes of the hypercube. The vertex set of the $n$-dimensional hypercube graph, denoted $Q_n$, may be naturally identified with the set of binary strings of length $n$, or equivalently with the integers
	\[
	[2^n]=\{0,1,\ldots,2^n-1\}
	\]
	via binary expansion: a binary string 
	\[
	(i_1,i_2,\ldots,i_n), \text{ where } i_j\in\{0,1\}
	\]
	corresponds to the integer
	\[
	x = 2^{i_n}+2^{i_{n-1}}+\cdots+2^{i_1}.
	\]
	
	For vertices $a,b\in Q_n$, viewed as binary strings, the \emph{Hamming distance} $d_H(a,b)$ is defined as the number of coordinates in which the two strings differ. This distance coincides with the shortest‑path distance in the $n$‑hypercube graph.
	
	Fix integers $n$ and $r\ge 0$ and recall that $G_{n,r}$ denotes the graph with vertex set $Q_n$ in which two vertices $a,b\in Q_n$ are adjacent if and only if $d_H(a,b)\le r$. By definition of diameter, the Vietoris--Rips complex $\VR$ is the clique complex of the graph $G_{n,r}$. In particular, $\VR$ is a flag
	complex, and all its minimal non‑faces are of dimension one.
	
	Using the labeling $Q_n=[2^n]$, the Stanley--Reisner ring of $\VR$ is therefore
	\[
	SR[\VR]
	=
	\Bbbk[x_0,\ldots,x_{2^n-1}]/\mathcal{I}_{\VR}
	\]
	where the Stanley--Reisner ideal $\mathcal{I}_{\VR}$ is generated by the quadratic monomials
	\[
	x_a x_b,
	\quad a,b\in Q_n \text{ such that } d_H(a,b)>r.
	\]	
	
	To obtain upper bounds on the connectivity of $\VR$, we study the Tor algebra 
	\[
	\Tor^{i,j}_{\Bbbk[x_0,\ldots,x_{2^n-1}]}(SR[\VR],\Bbbk)
	\]
	by resolving $SR[\VR]$ by free $\Bbbk[x_0,\ldots,x_{2^n-1}]$-modules known as the \emph{Taylor resolution}~\cite{t} (see also~\cite{jls}). The key feature of this resolution is that it is completely determined by a chosen generating set of the Stanley–Reisner ideal and therefore reflects the combinatorics of the missing faces of the simplicial complex $K$.
	
	We begin by fixing an ordering $\{m_1, \dots, m_r\}$ of monomial generators of the ideal $\mathcal{I}_K$. Let $\Lambda(e_1, \dots, e_r)$ denote the exterior algebra generated by symbols $e_1,\ldots, e_r$, where each $e_i$ corresponds bijectively to the monomial $m_i$. 
	
	The ordering on the generators of $\mathcal I_K$ induces an ordering on the generators $\{e_i\}$ and hence on all monomials in the exterior algebra. The generators of the Taylor complex are the monomials
	\[
	e_I=e_{i_1} e_{i_2}\cdots e_{i_k}
	\]
	where $I=(i_1,\ldots, i_t)$ is an increasing sequence with $i_1< \cdots < i_t$. The induced ordering on these generators is the standard lexicographic ordering: for two index sets $J=\{j_1,\ldots,j_s\}$ and $L=\{\ell_1,\dots,\ell_t\}$, we set that  $e_J < e_L$ if there exists an index $k$ such that $j_i=\ell_i$ for $i<k$ and $j_k < \ell_k$.  
	
	For a subset $J =\{j_1,\ldots, j_s\} \subseteq \{1,\ldots,r\}$, we denote by 
	\[
	m_J=\lcm(m_{j_1},\ldots,m_{j_s})
	\]
	the least common multiple of the corresponding monomials and write $|J|=s$. The \emph{degree} of $e_J$, denoted $deg(e_J)$, is defined to be the degree of the monomial $m_J$ (and not the dimension of a missing face).
	
	\begin{defin} \label{def:taylor} The \emph{Taylor resolution} of $SR[K]$ is the free $\kv$-resolution 
		\[
		\mathcal T=\kv\otimes_\Bbbk\Lambda(e_1,\ldots, e_r)
		\]
		equipped with the differential
		\begin{equation} \label{taylordiff} 
			\delta(e_J) = \sum_{i=1}^{|J|} (-1)^{i-1} \frac{m_J}{m_{J\setminus \{j_i\}}} \otimes  e_{J\setminus \{j_i\}}.
		\end{equation} 
	\end{defin}
	
	Taylor~\cite{t} (see also \cite{jls}) showed that $(\mathcal T, \delta)$ is a free resolution of $SR[K]$.  With respect to the grading conventions set earlier, the Taylor resolution is naturally bigraded, with
	\[
	\deg(x_i)=(0,1) \quad \text{and } \quad \deg(e_J)=(|J|,\deg(e_J)).
	\]
	Since $(\mathcal{T},\delta)$ resolves $SR[K]$, it follows that
	\[
	\Tor^{i,j}_{\kv}(SR[K],\Bbbk) \cong H^{i,j}( \mathcal{T} \otimes_{\kv }\Bbbk). 
	\]

	Moreover, the Taylor resolution admits the structure of a differential graded
	algebra. A multiplicative structure was constructed by Gemeda ~\cite{g}
	(see also~\cite[Proposition~31.3]{p}), and is given on generators by
	\[
	e_W \ast e_Q =
	\begin{cases}
		(-1)^{\sigma(W,Q)},
		\dfrac{m_W m_Q}{m_{W\cup Q}} e_{W\cup Q},
		& \text{if } W\cap Q=\varnothing\\
		0 & \text{otherwise}
	\end{cases}
	\]
	where $\sigma(W,Q)$ denotes the sign of the permutation required to reorder
	the concatenation $WQ$ into increasing order.	
	
	In analogy with the Koszul complex, it is convenient to decompose the Taylor
	complex according to monomial support.
	
	\begin{defin}\label{def:taylorsupport} 
		For a square free monomial $M\in \kv$, we define 		
		\[
		\mathcal{T}\otimes_{\kv}\Bbbk\Big|_M
		\]
		to be the subcomplex of $\mathcal{T}\otimes_{\kv}\Bbbk$ generated by all elements $e_J$ such that $m_J=M$.
	\end{defin} 
	
	By \cite[Theorem 3.2.9]{bp}, the cohomology of the suspension $\Sigma K$ is isomorphic to the cohomology of the restriction of the Taylor complex to elements with support $x_0\cdots x_{m-1}$ 
	\[
	\widetilde{H}^{m-i}(\Sigma K) \cong H^{-i,m}(\mathcal{T} \otimes _{\kv}\Bbbk\Big|_{x_0\cdots x_{m-1}}).
	\]
	
	We now specialise this construction to Vietoris--Rips complexes of the hypercube. In this case, the Stanley--Reisner ideal is generated by the quadratic monomials 
	\[
	\{x_ax_b\ |\ a<b, d_H(a,b)>r\}.
	\]  
	Accordingly, it is convenient to index the generators of the Taylor complex by unordered pairs $(a,b)$ and to write $e_{(a,b)}$ in place of $e_i$.
	
	We fix an ordering of these generators by setting that
	\[
	x_ax_b < x_cx_d
	\]
	if
	\begin{equation}\label{order}
		\begin{aligned}
			i)\quad &  d_H(a,b) < d_H(c,d) \\
			ii)\quad &  \text{the lexicographical ordering if }  d_H(a,b) = d_H(c,d). 
		\end{aligned} 
	\end{equation}	
	This ordering will be used throughout the remainder of the section.
	
	Finally, we recall a subcomplex of the Taylor
	complex introduced by Lyubeznik~\cite{l} (see also~\cite[section 3.2]{jls}) which is weakly equivalent to it and often substantially smaller. The Lyubeznik complex plays a central role in the proof of Theorem~\ref{thm:coconnect}.
	
	We begin by describing the \emph{Lyubeznik resolution} of $SR[\VR]$, whose definition depends explicitly on a fixed ordering of the monomial basis.  We shall use the ordering defined in~\eqref{order}.
	
	We now introduce the notion of admissibility, which is the key combinatorial ingredient in the construction of the Lyubeznik resolution. Informally, admissibility rules out those generators in the Taylor complex that are considered
	redundant by the presence of ``larger" monomial relations.
	
	\begin{defin}\label{rev}
		Let $I=(i_1,\ldots,i_t)$ with $1\le i_1<i_2<\cdots<i_t$.  
		We say that the generator $e_I$ is \emph{admissible} if, for every
		$1\le h\le t$ and every index $q>i_h$, the monomial
		\[
		m_q\notdivides \operatorname{lcm}(m_{i_1},\ldots,m_{i_h}).
		\]
	\end{defin}
	Thus admissibility requires that no generator indexed later in the chosen ordering divides the least common multiple determined by any initial segment of
	$I$. This condition depends crucially on the ordering of the monomial generators.	
	\begin{exam} 
		Consider the Taylor complex associated to $VR(Q_4;1)$. The element
		\[ 
		e_{(0,5)}e_{(1,2)}
		\]
		is not admissible since $x_2x_5\mid x_0x_1x_2x_5$  and  $e_{(2,5)} > e_{(1,2)}$. By contrast, the element
		\[
		e_{(0,3)}e_{(1,2)}
		\]
		is admissible since there is no generator $e_{(a,b)}$ with $d_H(a,b)>1$, $e_{(a,b)} > e_{(1,2)}$ and $x_ax_b\mid x_0x_1x_2x_3$.
	\end{exam}
	
	\begin{rmk}
		In \cite{jls}, admissibility is defined using smaller monomials, i.e.\ the
		condition $q>i_h$ is replaced by $q<i_h$. The variant used here, which tests
		divisibility against larger monomials, is sometimes referred to as
		\emph{reverse admissibility}; see, for example,~\cite{se}.
	\end{rmk}

	\begin{defin} 
		The \emph{Lyubeznik resolution} $\mathcal{L}$  is the subcomplex of the Taylor resolution $\mathcal{T}$ generated by all admissible $e_I$. 
	\end{defin}

	Lyubeznik~\cite{l} (see also~\cite[Section 3.2]{jls}) showed that $\mathcal{L}$ is a resolution of $SR[K]$. Specialising to $\VR$, this implies that 
	\[
	H(\mathcal{L} \otimes_{\kx}\Bbbk,\delta) \cong H(\mathcal{T}\otimes_{\kx} \Bbbk, \delta).
	\]
	We denote $\mathcal{T}\otimes_{\kx} \Bbbk $ by $\mathcal{T}_{n,r}$ and the corresponding \Lu subcomplex by $\mathcal{L}_{n,r}$. To summarize, we say the inclusion
	\[
	\iota\colon \mathcal{L}_{n,r} \lhook\joinrel\longrightarrow \mathcal{T}_{n,r}
	\]
	is a weak equivalence.
	
	Since $\VR$ is a flag complex, admissibility can be checked on pairs of generators. This leads to the following useful simplification.
	
	\begin{lem} \label{lem:inadmissible}
		If $e_I$ is not admissible, then there exist indices $s < w$ such that $e_{i_s} e_{i_w}$ is not admissible.		
	\end{lem}
	
	\begin{proof}  
		Since $e_I$ is not admissible, there exists a generator $e_{(u,v)}$ such that $x_ux_v > x_{a_{i_t}}x_{b_{i_t}}$ and $x_ux_v\divides\lcm(e_I)$. Consequently, there are indices $\alpha$ and $\beta$ such that
		$i_\alpha=(a_\alpha,b_\alpha)$ with $u\in\{a_\alpha,b_\alpha\}$ and
		$i_\beta=(a_\beta,b_\beta)$ with $v\in\{a_\beta,b_\beta\}$. Let
		$w=\max\{\alpha,\beta\}$ and $s=\min\{\alpha,\beta\}$. Then
		\[
		x_u x_v > x_{a_w}x_{b_w}
		\quad\text{and}\quad
		x_u x_v \mid \operatorname{lcm}(e_{i_s}e_{i_w}),
		\]
		showing that $e_{i_s}e_{i_w}$ is not admissible.
	\end{proof} 
	
	We will frequently use products of Taylor generators corresponding to pairs of vertices at maximal Hamming distance.
	\begin{defin}\label{def:Thetan}
		Define
		\[
		\Theta_n
		=
		\prod_{\substack{0\le i<j<2^n\\ i+j=2^n-1}} e_{(i,j)}.
		\]
	\end{defin}		
	Each factor $e_{(i,j)}$ corresponds to a pair of vertices of the hypercube
	satisfying $d_H(i,j)=n$. Writing $\overline{i}$ for the bitwise complement of
	$i$, or equivalently the antipodal vertex to $i$, we  express
	\[
	\Theta_n
	=
	\prod_{i=0}^{2^{n-1}-1} e_{(i,\overline{i})}.
	\]
	The support of $\Theta_n$ is $x_0\cdots x_{2^n-1}$, and its bidegree is
	$(-2^{n-1},\,2^n)$.
	
	\begin{exam}
		For $VR(Q_3;1)$,
		\[
		\Theta_3
		=
		e_{(0,7)}e_{(1,6)}e_{(2,5)}e_{(3,4)}.
		\]
	\end{exam}

	\begin{prop}\label{prop:thetaadmis}
		For any $r$, the class $\Theta_{n}$ is admissible in the Lybeznik resolution $\mathcal{L}_{n,r}$.
	\end{prop}
	\begin{proof}
		Suppose, for contradiction, that $\Theta_{n}$ is not admissible. By Lemma~\ref{lem:inadmissible}, there are two factors of $\Theta_{n,r}$, say $e_{(i_1,j_1)}$ and $e_{(i_2,j_2)}$, with $ e_{(i_1,j_1)} <  e_{(i_2,j_2)}$ and  $e_{(i_1,j_1)} e_{(i_2,j_2)} $ is not admissible.
		For this product to be inadmissible, there must be a generator $e_{(a,b)}$ of the Taylor complex satisfying $e_{(a,b) }> e_{(i_2,j_2)}$ and $x_ax_b\mid\lcm(x_{i_1}x_{j_1}, x_{i_2}x_{j_2})$. Since $d_H(i_2,j_2)=n$, it follows that $d_H(a,b)=n$ as well.   However, among integers less than $2^n$, pairs $(a,b)$ with $a+b=2^{n}-1$ are unique. Hence, we must have $e_{(a,b)} = e_{(i_2,j_2)}$ or $e_{(i_1,j_1)}$ which contradicts the strict inequality $e_{(a,b)} > e_{(i_2,j_2)}$.
	\end{proof}

		

		

	For the remainder of this section, we fix $n$ and $r$.  For $s\ge 0$, we denote by
	$\mathcal{L}^s$ the summand of $\mathcal{L}_{n,r}$ consisting of elements of
	filtration $s$, that is, products of exactly $s$ generators
	$e_i$.
	
	\begin{lem}\label{lem:decompostion}
		Let $e_I\in\mathcal{L}^s$ be an admissible element whose support is $\operatorname{supp}(e_I)=x_0\cdots x_{2^n-1}$.
		Then $e_I$ decomposes as
		\[
		e_I = e_J\Theta_n
		\]
		where $e_J\in\mathcal{L}^{s-2^{n-1}}$ is admissible.
	\end{lem}

	\begin{proof} Fix $0\leq a < 2^{n-1}$ and suppose that $e_{(a,\overline{a})}$ is not a factor of $e_I$.  Since $\lcm(e_I)= x_0 \cdots x_{2^n-1}$ both $x_a$ and $x_{\overline{a}}$ must appear among the factors of $e_I$ with Hamming distance less than $n$. In other words, there must be a subproduct $e_L$ of $e_I$ with  $L=\{(a_{\ell_1},b_{\ell_1}),\ldots, (a_{\ell_s},b_{\ell_s})\}$ for $\ell_1<\cdots< \ell_s$ such that $d_H(a_{\ell_i},b_{\ell_i})< n$ and $x_ax_{\overline{a}}\mid\lcm(e_L)$. Since generators corresponding to pairs with Hamming distance $n$ are larger than generators corresponding to smaller Hamming distance, the relation $x_ax_{\overline{a}}$ is greater than $x_{a_{\ell_s}}x_{b_{\ell_s}}$. This contradicts the admissibility of $e_I$. Hence, every factor $e_{(a,\overline a)}$ must occur in $e_I$, and the claimed decomposition follows.
	\end{proof}
	The decomposition in Lemma~\ref{lem:decompostion} has as a consequence that the differentials of such elements in the Taylor resolution admit a particularly simple description.
	
	\begin{lem}\label{lem:diff}
		Let $e_J\Theta_n $ be an admissible element of $\mathcal{L}_{n,r}$. Then its
		differential is 
		\[
		\delta(e_J\Theta_n) = \sum_{j_k \in J} (-1)^{k+1} e_{J\setminus \{j_k\}}\Theta_n.
		\]
	\end{lem}
	\begin{proof}
		The first terms of the Taylor differential~\eqref{taylordiff} are obtained by omitting a factor $e_{j_k}$. The corresponding coefficient $\epsilon(j_k)$  is nonzero when $\lcm(e_{J\setminus \{j_k\}} \Theta_n)=x_0\cdots x_{2^n-1}$, which always holds since $\lcm(\Theta_n)=x_0\cdots x_{2^n-1}$.
		
		All remaining terms in the Taylor differential arise by omitting one of the factors $e_{(a,\overline a)}$ of $\Theta_n$. Such terms have a vanishing coefficient unless
		\[
		\lcm\left(e_J\frac{\Theta_n}{e_{(a,\overline a)}}\right)
		\neq x_0\cdots x_{2^n-1}.
		\]
		However, this would require both $x_a$ and $x_{\overline a}$ to appear in $\operatorname{lcm}(e_J)$, contradicting the admissibility of $e_J\Theta_n$. Consequently, all such terms vanish and the stated formula
		follows.
	\end{proof}

	We now introduce a combinatorial dual object that allows us to convert algebraic information about admissible classes in the Lyubeznik resolution into topological information about the cohomology of $\VR$. The guiding principle is that minimal missing faces in $\VR$ correspond to maximal simplices in an appropriate dual complex, and vice versa. This duality enables us to control high‑degree cohomology of $\VR$ by studying the homology of a simpler simplicial complex.
	
	In particular, by combining this correspondence with the decomposition results of the previous subsection, we will obtain effective upper bounds on the coconnectivity of $\VR$.

	We now define a simplicial complex that is $2^{n-1}-1$ dual to $\VR$ which itself is a flag complex. It is worth nothing that this is not the Alexander dual of $\VR$.
	
	\begin{defin} Define the ordered simplicial complex $(\mathcal{C}_{n,r},\widehat{d})$ as
		follows. The \emph{vertices} of $\mathcal{C}_{n,r}$ are the generators $e_{(a,b)}$ such that $r < d_H(a,b) < n.$
		
		An ordered collection of vertices
		\[
		e_{(a_1,b_1)},\ldots,e_{(a_t,b_t)}
		\]
		spans a $(t-1)$-simplex of $\mathcal{C}_{n,r}$ if and only if the product $e_{(a_1,b_1)}\cdots e_{(a_t,b_t)}$ is admissible in the Lyubeznik resolution $\mathcal{L}_{n,r}$.

		The differential $\widehat{d}$ is induced by the Leibniz rule 
		\[
		\widehat{d}(e_Je_I ) =
		\widehat{d}(e_J) e_I +(-1)^{|J|} e_J\widehat{d}(e_I)
		\]
		whenever $e_Je_I\Theta_n$ is admissible.
		The \emph{simplicial chain complex}  of $\mathcal{C}$ is denoted 
		$\widehat{\mathcal{C}}_{n,r}$  with $n,r$ omitted if $n,r$ are fixed.		
		The associated simplicial chain complex of $\mathcal{C}_{n,r}$ is denoted $\widehat{\mathcal{C}}_{n,r}$.
		
		When $n$ and $r$ are fixed, we write $\mathcal{C}$ for $\mathcal{C}_{n,r}$ and $\widehat{\mathcal{C}}$ for its simplicial chain complex. 
	\end{defin}
	
	\begin{lem}
		The simplicial complex $\mathcal{C}_{n,r}$ is flag.
	\end{lem}	
	\begin{proof}	
		Minimal non‑faces are
		exactly the pairs
		\[
		\{e_{(a_1,b_1)}, e_{(a_2,b_2)}\}
		\]
		for which the product
		$e_{(a_1,b_1)}e_{(a_2,b_2)}$ is not admissible in
		$\mathcal{L}_{n,r}$.
	\end{proof}
	
	\begin{cor}\label{cor:dual}   
		There is a morphism of differential modules
		\[
		\iota\colon \mathcal{L}_{n,r} \longrightarrow\widehat{\mathcal{C}}
		\]					
		which induces an isomorphism 
		\[
		H^{2^{n-1}-t-1}(\VR)\cong H_{t-1}(\widehat{\mathcal{C}})
		\]
		for $t>1$.	
	\end{cor}
	
	\begin{proof}
		The morphism $\iota$ is defined on generators by
		\[
		e_J \Theta_n \longmapsto \begin{cases}
			e_J, & J\neq\varnothing,\\
			0, & J=\varnothing.
		\end{cases}
		\]
		It is an isomorphism on all components of positive filtration; a class of filtration $|J|+2^{n-1}$ in $\mathcal{L}_{n,r}$ is sent to a $(|J|-1)$-chain in $\widehat{\mathcal{C}}$.
		
		In particular, a cohomology  class in $ H^{2^{n-1}-t}(\Sigma\VR)=H^{2^n-(t+2^{n-1})}(\Sigma\VR)$ is represented by a cocycle of filtration $t+2^{n-1}$ in $\mathcal{L}_{n,r}$  which maps to a $(t-1)$-cocycle in $\widehat{\mathcal{C}}$. 
	\end{proof}

	To prove Theorem~\ref{thm:coconnect}, we introduce a further
	simplicial complex, denoted by $\mathcal{J}$,  which may be of independent interest. While $\mathcal{C}$ captures the full admissibility conditions arising from the Lyubeznik resolution, $\mathcal{J}$ retains only those constraints that are essential for controlling top degree cohomology. This relaxation will allow us to relate the cohomology of $\VR$ to the homology of a simpler combinatorial object, while still preserving sufficient structure to obtain sharp coconnectivity bounds.
	
	The construction of $\mathcal{J}$ parallels that of $\mathcal{C}$. We begin by
	introducing a weaker notion of admissibility.
	
	\begin{defin}
		Let
		$e_J, J=\{(a_1,b_1),\ldots,(a_t,b_t)\}$ be a monomial in the generators of the Taylor complex. We say that $e_J$ satisfies the \emph{modified admissibility condition} if
		\[
		x_a x_{\overline{a}} \notdivides x_{a_1}x_{b_1}\cdots x_{a_t}x_{b_t}
		\]
		for all $0\le a<2^{n-1}$.
	\end{defin} 
	The difference between $\mathcal{J}$ and $\mathcal{C}$ is that, in the former, we only insist on non-divisibility by relations which are factors of $\Theta_n$ rather than all higher relations. In this sense, $\mathcal{J}$ can be viewed as a coarse dual to $\VR$ that isolates the behavior of cohomology in the highest degrees.	
	\begin{defin}
		
		The ordered simplicial complex $(\mathcal{J}_{n,r},\widehat{d})$ is defined as
		follows. Its \emph{vertices} are the generators $e_{(a,b)}$ with $r < d_H(a,b) < n$.
		
		An ordered set of vertices
		\[
		e_J = \{e_{(a_1,b_1)},\ldots,e_{(a_t,b_t)}\}
		\]
		spans a $(t-1)$‑simplex of $\mathcal{J}_{n,r}$ if and only if $e_J$ satisfies the modified admissibility condition.
		
		The differential $\widehat{d}$ is induced by
		\[
		\widehat{d}(e_Je_I)
		=
		\widehat{d}(e_J)e_I
		+
		(-1)^{|J|} e_J\widehat{d}(e_I),
		\]
		whenever the product $e_Je_I$ satisfies the modified admissibility condition. When $n$ and $r$ are fixed, we write $\mathcal{J}$ for $\mathcal{J}_{n,r}$. The associated simplicial chain complex is denoted by $\widehat{\mathcal{J}}$.
	\end{defin}
	
	\begin{exam}\label{exm:31} The vertices of $\mathcal{J}_{3,1}$  are 
		\[
		\begin{array}{cccccc}
			e_{(0,3)} & e_{(0,5)} & e_{(3,5)} & e_{(0,6)} & e_{(3,6)} & e_{(5,6)} \\
			| & | & | & | & | & | \\
			e_{(1,2)} & e_{(1,4)} & e_{(1,7)} & e_{(2,4)} & e_{(2,7)} & e_{(4,7)}
		\end{array}
		\]
		
		The top and bottom rows each span a 5-simplex. There is an edge between every
		pair of vertices within each row, and certain vertical edges connect vertices in
		the top row to vertices in the bottom row. No other edges occur. For instance,
		there is no edge between $e_{(0,3)}$ and $e_{(1,4)}$ since $x_4=x_{\bar 3}$. As a result, $\mathcal{J}_{3,1}$ has the homotopy type of a wedge of five circles.  
	\end{exam}
	
	Note that Lemmas \ref{lem:decompostion} and \ref{lem:diff} continue to hold when the Lyubeznik admissibility condition is replaced by the modified admissibility
	condition.
	
	\begin{prop} \label{prop:summand} 
		For $t>1$, the group $H^{2^{n-1}-t-1}(\VR)$ is a direct summand of $H_{t-1}(\widehat{\mathcal{J}})$.
	\end{prop}   
	Thus, the rank of $H_{t-1}(\mathcal{J}_{n,r})$ provides an upper bound on the rank of $H^{2^{n-1}-t-1}(\VR)$.
	
	\begin{proof}
		For $t>0$, there is the embedding of 
		\[
		\iota\colon\widehat{\mathcal{J}}\lhook\joinrel\longrightarrow\mathcal{T}\qquad
		e_J \longmapsto e_J \Theta_n.
		\]
		There is also an inclusion of chain complexes
		\[
		\lambda\colon \widehat{\mathcal{C}} \lhook\joinrel\longrightarrow \widehat{\mathcal{J}}.
		\]
		By Lyubeznik's theorem, the composition $\iota \circ \lambda$ induces the identity in homology.
	\end{proof}

\begin{lem}
	\label{lem:Jflag}
	The simplicial complex $\mathcal{J}$ is the clique complex of a graph $\mathcal{G}$ whose vertex set is the vertex set of $\mathcal{J}$, with an edge between two vertices if and only if they satisfy the modified admissibility condition.
\end{lem}	
\begin{proof}
	The modified admissibility condition is pairwise in
	nature. In particular, if the monomial $e_J$ is not modified admissible,
	then there exist indices $i\neq j$ such that
	$x_ax_{\overline a}\mid x_{a_i}x_{b_i}x_{a_j}x_{b_j}$, which implies that the pair
	$\{e_{(a_i,b_i)},e_{(a_j,b_j)}\}$ is not modified admissible. 
	Consequently, all minimal non-faces of $\mathcal{J}$ have cardinality two.
	Equivalently, $\mathcal{J}$ is the clique complex of its $1$-skeleton, which
	we denote by $\mathcal{G}$. 
\end{proof}

As a consequence, upper bounds on the coconnectivity of $\VR$ can be obtained by estimating combinatorial invariants of the graph $\mathcal{G}$ and its complement $\mathcal{G}^c$. We now proceed to complete the proof of Theorem~\ref{thm:coconnect} by bounding the total domination number of $\mathcal{G}^c$.

\begin{proof}[Proof of Theorem~\ref{thm:coconnect}]
	As in the proof of Theorem~\ref{thm:main}, we estimate the total domination number of $\mathcal{G}^c$. The vertex set of $\mathcal{G}^c$ consists of all generators
	$e_{(a,b)}$ with $0\le a,b<2^n$ and $r<d_H(a,b)<n$. For each vertex $a\in Q_n$, there are
	\[
	S=\sum_{i=r+1}^{n-1}\binom{n}{i}
	\]
	vertices $b\in Q_n$ satisfying $d_H(a,b)>r$ (this sum is defined since
	$n>r+1$). Consequently, $2^n \sum_{i=r+1}^{n-1}\binom{n}{i}$ pairs $(a,b)$ satisfy the above condition, and since each unordered pair is counted twice, the graph $\mathcal{G}^c$ has $2^{n-1}S$ vertices.
	
	Each vertex $e_{(a,b)}\in \mathcal{G}^c$ is adjacent to all vertices of the form $e_{(\overline a,c)}$ with
	$d_H(\overline a,c)>r$, and to all vertices of the form $e_{(c,\overline b)}$
	with $d_H(c,\overline b)>r$. Each of these sets has cardinality $S$, and since
	$e_{(\overline a,\overline b)}$ belongs to both, the degree of every vertex of
	$\mathcal{G}^c$ is
	\[
	\Delta(\mathcal{G}^c)=2S-1.
	\]
	
	By Theorem~\ref{thm:hy}, the total domination number $\gamma_t(\mathcal{G}^c) >\frac{2^{n-1}S}{2S-1}$, where the the inequality is strict since the right-hand side is not an integer. By Theorem~\ref{thm:ch}, the independence complex $I(\mathcal{G}^c) =\mathcal{J}$ is therefore $(\frac{2^{n-2}S}{2S-1}-1)$-connected. Finally, Proposition~\ref{prop:summand} implies that $H^{\ast}(\VR)=0$ for all $\ast\ge 2^{n-1}-\frac{2^{n-2}S}{2S-1}-1$, which proves the theorem.
\end{proof}

\section{Koszul complex descriptions of cohomology classes in $\VR$}\label{sec:koszul}

In the previous sections, we developed a toric topological framework for studying Vietoris--Rips complexes via Stanley--Reisner rings, moment-angle
complexes, and Tor algebras. This framework enabled us to extract global topological information, such as connectivity and coconnectivity, from combinatorial data.

In this section, we use the same techniques in a more constructive way to describe explicit cohomology classes in Vietoris--Rips complexes. Our approach is based
on the Koszul resolution computing the Tor algebra
\[
\Tor_{\Bbbk[x_0,\ldots,x_{m-1}]}(SR[\VR],\Bbbk)
\]
which allows us to represent cohomology classes by concrete monomials and to detect them using full subcomplexes of~$\VR$, answering the question of Adams and Virk~\cite{av} on combinatorial realisibility of (co)homological classes as the boundary of a cross polytope.

From this perspective, the cohomology classes constructed by Adams and Virk~\cite{av} appear naturally as $*$-decomposable classes in the cohomology of the associated moment-angle complexes, and their propagation
phenomena admit a conceptual explanation in terms of Tor algebra products and Koszul generators. We also introduce \emph{ghost vertices}, which provide a flexible tool for embedding full subcomplexes into~$\VR$ and for proving linear independence of families of cohomology classes.

We begin with a general observation that characterises cocycles in the Koszul complex in terms of the combinatorics of the underlying simplicial complex.

\begin{prop}\label{prop:gen} 
	Let $K$ be a simplicial complex with vertex set $[m]$. A nonzero  monomial
	\[
	\alpha=u_{j_1}\cdots u_{j_t} x_{i_1}\cdots x_{i_s}   
	\]
	in the Koszul complex of $\Tor_{\Bbbk[x_0,\ldots,x_{m-1}]}(SR[K],\Bbbk)$, whose support is $[m]$, is a cocycle if and only if $\{i_1, \ldots, i_s\}$ is a maximal simplex in $K$. 
\end{prop}
\begin{proof}
	The differential of $\alpha$ is given by
	\[
	d(\alpha) = \sum_{k=1}^{t} \pm \frac{u_{j_1} \cdots u_{j_t}}{u_{j_k}} x_{j_k} x_{i_1} \cdots x_{i_s}. 
	\]
	Thus $d(\alpha)=0$ if and only if each monomial $\frac{ u_{j_1} \cdots u_{j_t}}{u_{j_k}} x_{j_k} x_{i_1} \cdots x_{i_s}$ vanishes, which is equivalent to $x_{j_k} x_{i_1} \cdots x_{i_s}\in\mathcal{I}_{K}$ for all $k$.  
	Since the support of $\alpha$ is $[m]$, this condition holds when  no additional vertex can be added to the simplex $\{i_1,\ldots, i_s\}$, that is, when $\{i_1,\ldots, i_s\}$ is a maximal simplex in $K$.
\end{proof}
We now specialise this description to the Vietoris--Rips complex $\VR$. Recall
that the Stanley--Reisner ideal of $\VR$ is minimally generated by
\[
\{x_a x_b \mid d_H(a,b)>r,\ a,b\in[2^n]\}.
\]

When $n=r+1$, the combinatorics of $VR(Q_{r+1};r)$ becomes particularly simple, and the structure of its full subcomplexes can be described explicitly.

\begin{prop} \label{prop:Qr} 
	The full subcomplexes of $VR(Q_{r+1};r)$ satisfy
	\[	 
	\Sigma VR(Q_{r+1};r)_J\simeq
	\begin{cases} 
		S^{j} &\text{if } J=\{n_1,\overline{n}_1,\ldots, n_j,\overline{n}_j\} \\
		& \hspace{1.2em}\text{with } n_i,\overline{n}_i \text{ distinct integers in }
		[2^{r+1}],\\[0.3em]
		\ast
		& \text{otherwise}
	\end{cases} 
	\]
	where $\overline{n}$ denotes the integer obtained by switching the entries in the binary expansion of $n$.
	
	Moreover, if $J=\{ n_1,\overline{n}_1,\ldots, n_j,\overline{n}_j\} , L=\{ m_1,\overline{m}_1,\ldots, m_s,\overline{m}_s\}$ with all $n_i,m_i,\overline{n}_i, \overline{m}_i$ distinct, then the $*$-product~\eqref{starp} 
	\[
	\widetilde{H}^j(\Sigma VR(Q_{r+1};r)_J) \otimes \widetilde{H}^s(\Sigma VR(Q_{r+1};r)_L) \to\widetilde{H}^{j+s}(\Sigma VR(Q_{r+1};r)_{J \cup L};\Bbbk)
	\]
	is the isomorphism $\Bbbk\otimes_{\Bbbk}\Bbbk\to\Bbbk$.
\end{prop} 

\begin{proof} 
	Since antipodal vertices in $Q_{r+1}$ are at Hamming distance $r+1$, the complex $VR(Q_{r+1};r)$ is combinatorially the boundary of a cross polytope and can be expressed as the join of $2^r$ 0-spheres,
	\[
	VR(Q_{r+1};r) = \overbrace{ \{0,\overline{0}\}  
		\ast \cdots\ast \{ 2^r-1, \overline{2^r-1} \}}^{2^r}.
	\]
	A full subcomplex corresponding to a union of antipodal pairs is therefore the
	join of 0‑spheres, hence a sphere, whereas any subcomplex
	containing a single non‑antipodal vertex is a join with at least one factor a point and therefore it is contractible.
	
	In the Koszul complex for $VR(Q_{r+1};r)$, a generator of $\widetilde{H}^j(\Sigma(VR(Q_{r+1};r);\Bbbk) =\widetilde{H}^j(S^j;\Bbbk)$ is represented by the monomial
	\[
	\prod_{i=1}^j u_{\overline{n}_i}x_{n_i}.
	\]
	Applying the $*$-product in the Koszul complex~\eqref{starp},
	\[
	\prod_{i=1}^j u_{\overline{n}_i}x_{n_i} \otimes   \prod_{i=1}^s u_{\overline{m}_i}x_{m_i} \longmapsto\prod_{i=1}^j u_{\overline{n}_i}x_{n_i} \prod_{i=1}^s u_{\overline{m}_i}x_{m_i} 
	\]
	produces the generator of $\widetilde{H}^{j+s}(\Sigma VR(Q_{r+1};r)_{J\cup L};\Bbbk)$, proving the second
	claim.
\end{proof}

The constructions above allow us to identify explicit cohomology classes arising
from full subcomplexes of $\VR$. To complete the picture, we now address the problem of proving linear independence of certain families of classes in $H^*(\VR)$. Our strategy is to detect these classes by restricting them to
carefully chosen full subcomplexes $\VR_I$, where their behaviour becomes transparent.

To consider a full subcomplex $\VR_I$ as a complex on $2^n$ vertices, we add \emph{ghost vertices}, that is, vertices which are present algebraically but do not appear as $0$-simplices. This produces a simplicial complex
$\overline{VR}(Q_n;r)_I$ which admits a canonical inclusion into $\VR$ while retaining
precise control over its Koszul and Tor algebra structure.

For our purposes, ghost vertices are most conveniently described using the
Stanley--Reisner ring of $\overline{VR}(Q_n;r)_I$. This ring is defined by
\[
SR[\overline{VR}(Q_n;r)_I] = \Bbbk[x_0,\cdots,x_{2^n-1}]/ \mathcal{I}
\]
where the ideal $\mathcal{I}$ is generated by 
\[		  
\begin{array}{llll}
	x_{a}x_{b} &\text{if}  & a,b \in I & d_H(a,b)>r \\
	x_{a} & \text{if}  &  a\in [2^{n}]\setminus I.
\end{array}
\]	
There is a natural simplicial map
\[
\overline{VR}(Q_n;r)_I  \longrightarrow \VR
\]
of complexes on the same vertex set, inducing a morphism of corresponding Stanley--Reisner rings.

From the description of the Koszul complex, it follows that the reduced cohomology of $\overline{VR}(Q_n;r)_I$ is given by	
\begin{equation}\label{coghost}  
	\widetilde{H}^{\ast}(\overline{VR}(Q_n;r)_I)\cong\widetilde{H}^{\ast}(\VR_I)\otimes \Lambda( u_a\mid a \notin I).
\end{equation}

We conclude this section by illustrating the effectiveness of the Koszul complex approach through a reproof of two propagation type results of Adams and Virk~\cite{av}. From the toric topological viewpoint developed above, both results arise naturally from explicit Koszul representatives and their behaviour under inclusions of full subcomplexes. In particular, the ghost vertex construction provides a conceptual explanation of the resulting propagation phenomena.

We begin by recalling first a lower bound on the rank of the homology of $\VR$ in
degree $2^r-1$, and then state a general propagation result describing how
non‑trivial homology classes arising on smaller hypercubes extend to larger
ones.

\begin{thm}[Adams-Virk~\cite{av}, Theorem~4.1]
	\label{thm:4.7}
	For $n\ge r+1$,		
	\[
	\rank H_{2^r-1}(\VR)\ge 2^{n-(r+1)} \binom{n}{r+1}.
	\]
\end{thm}

\begin{thm}[Adams--Virk~\cite{av}, Theorem~6.4]
	\label{thm:4.8}
	Let $q\ge 1$ and let $p$ be the smallest integer such that $H_q(VR(Q_p;r))\neq 0$. Then, for $n\geq p$,
	\[
	\rank H_q(\VR)\ge\sum_{i=p}^n 2^{i-p} \binom{i-1}{p-1}  \rank H_q(VR(Q_p;r)).
	\]
\end{thm}
From the Koszul complex viewpoint, the propagation phenomenon described in Theorems~\ref{thm:4.7} and~\ref{thm:4.8} reflect the fact that a nonzero Tor algebra class detected on $VR(Q_p;r)$ can be extended to
larger vertex sets by adding ghost vertices. Each such extension produces a new $*$-decomposable Koszul cocycle with full support. The combinatorial
choices involved account for the multiplicative factor
$2^{\,i-p}\binom{i-1}{p-1}$  appearing in the Adams–Virk bounds, while the ghost vertex construction ensures that the resulting classes remain linearly independent in cohomology, thereby establishing the
claimed lower bound.

Adams and Virk obtain their results by constructing isometric copies of $Q_p$ in $Q_n$ and then showing that the certain embeddings induce injections on homology by  constructing splitting maps from $Q_n$ onto $Q_p$. In~\cite{av}, these splittings are constructed geometrically: for Theorem~\ref{thm:4.7} they are called contractions, while for Theorem~\ref{thm:4.8} they are called concentrations. In contrast, our approach replaces these splitting maps by algebraic detection of full subcomplexes of $\VR$ realising the copies of $Q_p$. The Koszul complex then provides a natural
setting in which restriction maps and linear independence can be analyzed uniformly.

We now recall the embeddings of $Q_p$ into $Q_n$ described in
\cite[Section~2.4]{av}, and reinterpret them in terms of full subcomplexes of
$\VR$.  Let $S = \{s_1,\cdots,s_p\} \subseteq [n]$ be a set of cardinality $p$, referred to as the set of variable coordinates. Fix $\mathbf{b}=\{b_i\}_{i \in [n]\setminus S}$ of values to the remaining coordinates; following~\cite{av}, this choice is called an \emph{offset}. With this data we associate a full subcomplex of $\VR$. Let $J_{\mathbf{b}}\subseteq Q_n$ consist of all vertices
with binary expansions
\[
(j_{n-1},\ldots,j_0),\qquad j_i\in\{0,1\}
\]
where
\[
j_i=
\begin{cases}
	0 \text{ or } 1, & i\in S,\\
	b_i, & i\notin S.
\end{cases}
\]  

The full subcomplex $\VR_{J_{\mathbf{b}}}$ is isomorphic to $Q_p$ via the map $Q_p\longrightarrow \VR_{J_{\mathbf{b}}}$ sending a vertex $\{a_1, \ldots,a_p\} \in Q_p$ to the vertex of $Q_n$ with coordinates
\[	  
j_i= 
\begin{cases} 
	a_i, &  i\in S \\ 
	b_i,  & i\notin S.
\end{cases} 
\]
We now sketch the algebraic proof of Theorem~\ref{thm:4.7}; the argument for
Theorem~\ref{thm:4.8} follows the same pattern, with additional bookkeeping.	
\begin{proof}[Proof of Theorem~\ref{thm:4.7}] Let $p=r+1$. By~\eqref{coghost},  the degree $2^r$ cohomology of $\Sigma\overline{VR}(Q_n;r)_{J_{\mathbf{b}}}$ is isomorphic to $\mathbb Z$ with generator represented by 
	\[
	\Bigl[\prod_{i\in [2^n]\setminus J_{\mathbf{b}} }u_{i}  \Bigr]\cdot\Bigl[\prod u_{\overline{n}_i} x_{n_i} \Bigr]  
	\]
	where $\{ n_1 ,\ldots, n_{2^r} , \overline{n}_1, \ldots, \overline{n}_{2^r}\} =J_{\mathbf{b}}$ and $\{ n_1 ,\ldots, n_{2^r}\}$ is a maximal simplex in $\Sigma\VR_{J_{\mathbf{b 
	}}}$.  
	Adams and Virk give a condition on such maximal simplices that guarantees that the corresponding class lift to $H^{2^r}(\Sigma\VR)$. This condition can be expressed purely in terms of the Stanley–Reisner ideal.	
	
	\begin{defin}\label{defi:localdim}   
		A maximal simplex $\{n_1,\ldots, n_{2^r}\}$ satisfies the
		\emph{local diameter} condition if, for every $t$ with $1 \le t\le 2^r$, there is an index $i$ such that 
		$x_{n_t}x_{n_i}\in\mathcal I_{VR(Q_{r+1};r-1)}$.
	\end{defin}

	\begin{prop} \cite[Proposition 4.4]{av} \label{prop:lift} 
		If $\{n_1, \ldots, n_{2^r}\}$  satisfies the local diameter condition, then  
		\[
		\prod_{i \in [2^r] \setminus \{n_1, \ldots, n_{2^r}\}} u_i \cdot \prod x_{n_1}\cdots x_{n_{2^r}}
		\]
		is a cocycle in the Koszul complex for $\VR$.
	\end{prop} 
	Adams and Virk show that, for each choice of $J_{\mathbf{b}}$, there exists a maximal simplex satisfying the local diameter condition. Varying the choice of $S$ and the offset $\mathbf{b}$ produces $2^{n-(r+1)}\binom{n}{r+1}$ such cocycle representatives of classes in $H^{2^r}(\Sigma\VR)$.
	
	To prove linear independence, one further chooses the classes so that each cocycle fails to exist on
	any smaller cube (known as the cubic hull condition in~\cite{av}). Denote these cocycles by $c_1, \dots, c_k$ where $k=2^{n-(r+1) }\binom{n}{r+1}$.  Suppose that
	$\lambda= \sum_{i=1}^k \lambda_i c_i=0$. Mapping the
	Koszul complex of $\Sigma\VR$ to the Koszul complexes of
	$\Sigma\overline{VR}(Q_n;r)_{J_{\mathbf{b}}}$ for various $\mathbf{b}$, then detects individual generators, forcing all coefficients $\lambda_i$ to vanish. This completes the proof of Theorem~\ref{thm:4.7}.	
\end{proof}

\begin{proof}[Proof of Theorem~\ref{thm:4.8}]
	The proof follows the same strategy as that of Theorem~\ref{thm:4.7} and proceeds by induction on $n$. The base case $n=p$ is immediate from the assumption that
	$H_q(VR(Q_p;r))\neq 0$.
	
	For the inductive step, the sets $J_{\mathbf{b}}$ are defined as in~\cite[Section~5]{av} and determine the corresponding full subcomplexes $\VR_{J_{\mathbf{b}}}$. Introducing ghost vertices yields complexes
	$\overline{VR}(Q_n;r)_{J_{\mathbf{b}}}$ admitting canonical maps into $\VR$. The
	induced maps between Koszul complexes are precisely the concentration maps of~\cite[page~72]{av}. These extend non‑trivial Koszul cocycles from $VR(Q_p;r)$ to $\VR$ while preserving linear independence. Counting the admissible choices of $J_{\mathbf{b}}$ gives the stated lower bound.
\end{proof}

\begin{exam}
	We illustrate the effectiveness of the Koszul complex by computing
	the rank of $H^3(VR(Q_4;2))$. Adams and Virk~\cite{av} show in
	Theorem~\ref{thm:4.7} that $VR(Q_4;2)$ has at least eight linearly independent cohomology classes in degree 3. An additional ninth generator is constructed geometrically in~\cite[Section~7]{av}, although there is no explicit discussion that it is linearly independent from the eight other classes. Using the Koszul complex, we recover all nine classes and verify their linear independence in a uniform algebraic way.
	
	More precisely, we exhibit nine cocycles in the Koszul complex of $VR(Q_4;2)$ representing linearly independent classes in $H^4(\Sigma VR(Q_4;2))$. These cocycles are
	\begin{equation}  \label{cocycles} \end{equation}
	
	\[
	\begin{aligned}
		a_1 &= u_\ast x_0x_1x_2x_3, &
		a_2 &= u_\ast x_0x_1x_4x_5, &
		a_3 &= u_\ast x_0x_2x_4x_6, &
		a_4 &= u_\ast x_0x_3x_5x_6,\\
		a_5 &= u_\ast x_2x_3x_6x_7, &
		a_6 &= u_\ast x_1x_3x_5x_7, &
		a_7 &= u_\ast x_1x_2x_4x_7, &
		a_8 &= u_\ast x_4x_5x_6x_7,\\
		a_9 &= u_\ast x_8x_9x_{10}x_{11}.
	\end{aligned}
	\]
	Here $u_\ast$ denotes the product of the $u_i$ corresponding to the vertices complementary to those appearing in the $x$–monomial.
	
	Each class $a_1,\ldots, a_8$ is of the form $u_*x_{n_1}x_{n_2}x_{n_3}x_{n_4}$, where $u_* x_{n_1}x_{n_2}x_{n_3}x_{n_4}$ lies in the image of $VR(Q_3;2)\longrightarrow VR(Q_4;2)$ that inserts a $0$ into the leftmost coordinate of each vertex $i$. In the terminology of~\cite{av}, this corresponds to an offset of $0$ in the left position. Among the sixteen  $3$-dimensional simplices of $VR(Q_3;2)$, there are the eight  simplices $\{n_1,n_2,n_3,n_4\}$ which satisfy  Definition~\ref{defi:localdim}.    Therefore, by Proposition~\ref{prop:lift}, each class $a_i$ for $i=1,\dots, a_8$ represents a cocycle. The class $a_9$ arises as the image of the simplex $\{0,1,2,3\}$ with offset $1$ in the leftmost coordinate.
	
	Adams and Virk impose an additional \emph{cubical hull} condition on the maximal simplices in $VR(Q_3;2)$ in order to prove that the corresponding classes in $VR(Q_4;2)$ are lineraly independent, see~\cite[Lemma 4.5]{av}. This cubical hull condition is not satisfied by the classes in~\eqref{cocycles}. Instead we establish their independence by introducing ghost vertices.
	Suppose that
	\[
	R=  \sum_{i=1}^9 \lambda_i a_i = 0
	\quad\text{in } H^4(\Sigma VR(Q_4;2)).
	\]
	We show that all coefficients $\lambda_i$ must vanish. We do this by finding  full subcomplexes $VR(Q_4;2)_{L_i}$ $ i= 1, \dots, 9$ of $VR(Q_4;2)$ which are of the form $S^0 \ast S^0 \ast S^0 \ast S^0$. Furthermore, for each $i$, the map $\overline{VR}(Q_4;2)_{L_i}\longrightarrow VR(Q_4;2)$ induces a map in the Koszul complex which sends $a_i$ non-zero and $a_j, j \neq i $ to zero which implies $\lambda_i=0$.
	
	First, consider the set
	\[
	L_2=\{0,1,4,5,8,9,12,13\}.
	\]
	In the Koszul complex for the ghost complex $\overline{VR}(Q_4;2)_{L_2}$, the variables $x_a$ with $a\in[16]\setminus L_2$ map to zero. Consequently,
	\[
	a_i \longmapsto 0 \quad \text{for all } i \neq 2.
	\]

	The full subcomplex $VR(Q_4;2)_{L_2}$  is combinatorially of the form
	\[ 
	\{0,13\} \ast \{1,12\} \ast \{4,9\} \ast \{5,8\}
	\]
	which is the boundary of a cross polytope. In particular $a_2$ maps to the generator of $H^4(\Sigma \overline{VR}(Q_4;2)_{L_2})$. It follows that $\lambda_2=0$. Repeating this argument with appropriate choices of $L_i$ shows that
	\[
	\lambda_i=0 \qquad \text{for } i=2,\ldots,9.	
	\]
	(For example, letting $L_9=\{8,9,10,11,15,14,13,12\}$,  $a_9$ maps to the generator of $H^4(VR(Q_4;2)_{L_9})$ and $\lambda_9=0$.)  So $R=\lambda_1 a_1=0$ implying that $\lambda_1=0$.
	
	This proves that the nine classes $\{a_1,\ldots,a_9\}$ are linearly independent,
	and hence that $\rank H^3(VR(Q_4;2))\ge 9$. It is shown in~\cite{aa} that $\rank H^3(VR(Q_4;2))= 9$.
\end{exam}





	
		
		


Finally, we show that the classes constructed by Adams and Virk arise naturally as $*$-decomposable elements in the cohomology of the associated moment-angle complex, providing a conceptual explanation of their structure and propagation behaviour.		
\begin{thm} 
	The duals of the classes in~\cite[Theorem 4.1]{av} are the $\ast$--product of one dimensional classes. 
\end{thm} 
\begin{proof}
	In terms of the Koszul complex the duals of the cycles representing the classes in~\cite[Theorem 4.1]{av} are of the form  $u_* x_{n_1} x_{n_2} \cdots x_{n_{2^r}}$. Specifically, the simplex  $\{n_1, n_2, \dots, n_{2^r}\}$ is maximal in $VR(Q_n;r)$. Hence, for each $x_{n_i}, 1 \leq i \leq 2^r$, there are classes $u_{m_1}\cdots u_{m_s}$ with $d_H(n_i,m_j)>r$.
	
	Consider the one dimensional cocycles 
	\[
	u_{m_1} \cdots u_{m_s} x_{n_i},\quad   i=1, \dots, 2^r
	\]  
	in the Koszul complex. For different $i$, a class $u_{m_j}$ may appear more than once. By removing redundancies, we are left with a collection of one dimensional cocycles with disjoint support
	\[   
	u_{m_{i1}} \cdots u_{m_{i s_i}}x_{n_i}  \in VR(Q_n;r)_{J_i}, \quad i=1,\dots, 2^r   
	\]
	where $VR(Q_n;r)_{J_i}$ is the full subcomplex with $J_i=\{m_{i1}, \dots, m_{i s_i}, \  n_i \}$.  Note that $\{m_{i1}, \dots, m_{i s_i}\}$ may be empty.
	The class $ u_* x_{n_1} x_{n_2} \cdots x_{n_{2^r}}$ is the product in the Koszul complex of these one dimensional classes.
	
	Equivalently, we can say that the duals of the cycles described in \cite{av} are of the form
	$\alpha_{J_1} \ast \cdots \ast \alpha_{J_w}$ where the $J_i$ are mutually disjoint and $\bigcup J_i=[2^n]$. 
\end{proof}

\section{Combinatorial representative of a class in $H_4(VR(Q_5;3))$}
In this section we give an explicit combinatorial representative of the unique
non‑trivial cohomology class in $H^{4}(VR(Q_{5};3);\mathbb{Z})$. Although it is
known that
\[
H^{4}(VR(Q_{5};3);\mathbb{Z}) \cong \mathbb{Z}
\]
no explicit combinatorial representative of this class appears to have been
written down in the literature.

While it is often relatively straightforward to write down explicit cycles in the homology of a Vietoris--Rips complex, proving that such cycles represent non-trivial homology classes is considerably more subtle. At present, there is no systematic method for doing so when the cycle does not contain a maximal simplex. This is precisely the situation for $VR(Q_5;3)$, which has no maximal simplices of dimension four. We therefore begin by listing several families of explicit $4$-dimensional cycles and then prove that they represent a non-trivial homology class by constructing a cocycle that detects them.

To construct $4$-dimensional cycles in $VR(Q_{5};3)$, we observe that if missing
edges correspond to pairs of vertices at Hamming distance 5, such
pairs generate $S^{0}$-factors, then joins of five disjoint such $S^{0}$’s form
boundaries of $5$-dimensional cross polytopes, hence embedded $4$-spheres.

In $VR(Q_{5};3)$, disjoint pairs of vertices at distance five have the forms $\{v,v^{12345}\},
\{v^{i},v^{jklm}\}$, $\{v^{ij},v^{klm}\}$,
where $v^{i_1\ldots i_k}$ denotes the vertex whose binary expansion has $1$’s
precisely in the positions $\{i_1,\ldots,i_k\}$.

This leads to the following three families of embedded $4$-spheres:
\[
\begin{aligned}
	(A)\;&
	\{v^i,v^{jklm}\}*\{v^j,v^{iklm}\}*\{v^k,v^{ijlm}\}*
	\{v^l,v^{ijkm}\}*\{v^m,v^{ijkl}\},\\
	(B)\;&
	\{v,v^{12345}\}*\{v^{ij},v^{klm}\}*\{v^{ik},v^{jlm}\}*
	\{v^{il},v^{jkm}\}*\{v^{im},v^{jkl}\},\\
	(C)\;&
	\{v^i,v^{jklm}\}*\{v^j,v^{iklm}\}*
	\{v^{kl},v^{ijm}\}*\{v^{km},v^{ijl}\}*\{v^{lm},v^{ijk}\}.
\end{aligned}
\]
Each join is the boundary of a $5$-dimensional cross polytope and hence consists
of $32$ oriented $4$-simplices. Exactly one simplex in each boundary has all
vertices indexed by numbers less than $16$. We call this simplex the
\emph{small chain} of the corresponding cycle.

For example,
\[
\{1,30\}*\{2,29\}*\{4,27\}*\{8,23\}*\{15,16\}
\]
has small chain $[1,2,4,8,15]$.

A $4$-dimensional cocycle  
{\footnotesize
	\[
	\begin{aligned}
		\alpha = &[3,4,8,13,15],\ \textcolor{red}{[1,6,10,12,15]},\ \textcolor{red}{[2,5,9,12,15]},\ [1,2,7,12,15],\ [1,2,6,12,15],\\[0.4em]
		&[1,2,5,12,15],\ [1,2,4,12,15][1,6,8,11,15],\ [2,5,8,11,15],\ \textcolor{red}{[3,4,9,10,15]},\\[0.4em]
		&[1,6,8,10,15],\ [1,4,7,10,15],\ [1,4,6,10,15],\ [1,3,4,10,15],\ [1,2,4,10,15],\ [2,5,8,9,15],\ [3,4,8,9,15],\\[0.4em]
		&[2,4,7,9,15],\ [2,4,5,9,15][2,3,4,9,15],\ [1,2,4,9,15],\ \textcolor{red}{[3,5,6,8,15]},\\[0.4em]
		&[1,3,6,8,15],\ [1,2,6,8,15],\ [3,4,5,8,15][2,3,5,8,15],\ [1,2,5,8,15],\ [1,3,4,8,15],\\[0.6em]
		&\textcolor{red}{[1,2,4,8,15]},\ \textcolor{red}{[0,7,11,13,14]},\ \textcolor{red}{[3,4,8,13,14]},\ [0,3,7,13,14],\ [0,3,6,13,14],\ [0,3,5,13,14],\ [0,3,4,13,14],\\[0.4em]
		&[2,5,9,12,14],\ [0,7,9,11,14],\ \textcolor{red}{[2,5,8,11,14]},\ [0,5,7,11,14],\ [0,5,6,11,14],\ [0,3,5,11,14],\ [0,2,5,11,14],\\[0.4em]
		&[0,7,9,10,14],\ [3,4,9,10,14],\ [2,5,8,9,14],\ [3,4,8,9,14],\ \textcolor{red}{[2,4,7,9,14]},\ [0,3,7,9,14],\ [0,2,7,9,14],\ [2,4,5,9,14],\\[0.4em]
		&\textcolor{red}{[0,3,5,9,14]},\ [0,2,5,9,14],\ [2,3,4,9,14],\ [0,3,4,9,14],\ [3,5,6,8,14],\ [3,4,5,8,14],\ [2,3,5,8,14],\\[0.6em]
		&[0,3,5,8,14],\ [0,7,11,12,13],\ [1,6,10,12,13],\ [0,7,10,11,13],\ \textcolor{red}{[1,6,8,11,13]},\\[0.4em]
		&[0,6,7,11,13],\ [0,5,6,11,13],\ [0,3,6,11,13],\ [0,1,6,11,13],\ [0,7,9,10,13],\ [3,4,9,10,13],\ [1,6,8,10,13],\\[0.4em]
		&[3,4,8,10,13],\ \textcolor{red}{[1,4,7,10,13]},\ [0,3,7,10,13],\ [0,1,7,10,13],\ [1,4,6,10,13],\ \textcolor{red}{[0,3,6,10,13]},\ [0,1,6,10,13],\\[0.4em]
		&[1,3,4,10,13],\ [0,3,4,10,13],\ [3,5,6,8,13],\ [3,4,6,8,13],\ [1,3,6,8,13],\ [0,3,6,8,13],\ [1,6,10,11,12],\\[0.4em]
		&[0,7,9,11,12],\ [2,5,9,11,12],\ [1,6,8,11,12],\ [2,5,8,11,12],\\[0.4em]
		&[0,5,7,11,12],\ \textcolor{red}{[1,2,7,11,12]},\ [0,1,7,11,12],\ \textcolor{red}{[0,5,6,11,12]},\ [1,2,6,11,12],\\[0.4em]
		&[0,1,6,11,12],\ [1,2,5,11,12],\ [0,2,5,11,12],\ \textcolor{red}{[0,7,9,10,12]},\ [1,6,7,10,12],\ [1,4,7,10,12],\ [1,2,7,10,12],\\[0.4em]
		&[0,1,7,10,12],\ [2,5,7,9,12],\ [2,4,7,9,12],\ [1,2,7,9,12],\ [0,2,7,9,12],\ [3,5,6,8,11],\ [2,5,6,8,11],\ [1,5,6,8,11],\\[0.4em]
		&[3,4,7,9,10],\ [2,4,7,9,10],\ [1,4,7,9,10],\ [0,4,7,9,10]
	\end{aligned}
	\]}
was constructed computationally in Ripser. The cocycle $\alpha$ is a sum of
cochains dual to simplices whose vertices are all less than $16$. In particular,
the duals of all small chains appearing in the cycles of types $(A)$, $(B)$, and
$(C)$ occur as summands of $\alpha$ (shown in red).

Consequently, for any such cycle $\beta$ we have
\[
\langle \beta,\alpha\rangle \neq 0,
\]
which shows that all cycles arising from $(A)$, $(B)$, and $(C)$ are non‑zero in
homology and, moreover, represent the same generator of
$H_{4}(VR(Q_{5};3))\cong\mathbb{Z}$.

This provides an explicit combinatorial realisation of the unique $4$–dimensional
homology class of $VR(Q_{5};3)$. 

\begin{cor}
	The 4-degree cohomological class of $VR(Q_5;3)$ can be represented by the boundary of a cross polytope.
\end{cor}
\bibliographystyle{amsalpha}
\bibliography{references}

\end{document}